\documentclass{amsart}
\usepackage{amsmath}
\usepackage{amssymb}
\usepackage[all]{xy}
\usepackage{comment}
\usepackage{color}

\theoremstyle{plain}
\newtheorem{thm}{Theorem}[section]
\newtheorem{thm*}{Theorem}[section]
\newtheorem{cor}[thm]{Corollary}
\newtheorem{prop}[thm]{Proposition}
\newtheorem{lemma}[thm]{Lemma}

\newtheorem{lemma*}{Lemma}

\newtheorem{construct}[thm]{Construction}

\theoremstyle{definition}
\newtheorem{defn}[thm]{Definition}
\newtheorem{remark}[thm]{Remark}

\newtheorem*{remark*}{Remark}

\newtheorem{ex}[thm]{Example}
\newtheorem{notation}[thm]{Notation}

\newtheorem{question*}{Question}

\numberwithin{equation}{thm}

\newcommand{\bR}{\mathbb R}

\newcommand{\bC}{\mathbb C}

\newcommand{\cB}{\mathcal B}

\newcommand{\cN}{\mathcal N}

\def\Spec{\operatorname{Spec}\nolimits}

\newcommand{\bG}{\mathbb G}

\newcommand{\bA}{\mathbb A}

\newcommand{\cP}{\mathcal P}
\newcommand{\cF}{\mathcal F}
\newcommand{\bP}{\mathbb P}
\newcommand{\bZ}{\mathbb Z}
\newcommand{\bF}{\mathbb F}

\newcommand{\cC}{\mathcal C}

\newcommand{\cE}{\mathcal E}

\newcommand{\cJ}{\mathcal J}

\newcommand{\ol}{\overline}
\newcommand{\ul}{\underline}

\def\Spec{\operatorname{Spec}\nolimits}
\def\sl2{\operatorname{SL_{2(2)}}\nolimits}
\def\Ga2{\operatorname{\mathbb G_{\rm a(2)}}\nolimits}

\def\GL{\operatorname{GL}\nolimits}

\newcommand{\bN}{\mathbb N}
\newcommand{\bQ}{\mathbb Q}

\newcommand{\bu}{\bullet}

\date\today

\begin{document}

 \title[The Stable Adams Conjecture]{The Stable Adams Conjecture}
 
 \author[ Eric M. Friedlander]
{Eric M. Friedlander$^{*}$} 

\address {Department of Mathematics, University of Southern California,
Los Angeles, CA 90089}
\email{ericmf@usc.edu}

\thanks{$^{*}$ partially supported by the Simons Foundation }

\subjclass[2020]{55N20, 55R15}

\keywords{stable Adams conjecture, spectra, simplicial schemes}

\begin{abstract}
We provide proofs of several variants for spectra  of the 
original Adams Conjecture.  Our approach uses $\cF$-spaces (also
known as Segal $\Gamma$-spaces), resulting in establishing 
both unstable and stable homotopy equivalences between maps 
of spectra associated to the ${\bf J}$-map and the Adams operations.  We employ 
a rigid version of Artin-Mazur \'etale homotopy theory applied
to commutative diagrams of simplicial schemes defined over the Witt
vectors of fields of positive characteristic.  Much of this text is devoting
to developing a theory of $X$-fibrations and $\bZ/\ell$-complete
$X$-fibrations over the $\cF$-spaces we consider.
\end{abstract}

\maketitle


\section{Introduction}

The original Adams Conjecture \cite{Adams} compared the 
$J$-homomorphism from topological $K$-groups to the stable homotopy groups of 
spheres to the composition of this  $J$-homomorphism with the Adams operation
$\psi^p$ on $K$-theory for some prime $p$.   These homomorphisms are 
induced by maps from the topological K-theory spectrum to the delooping 
of the sphere spectrum.

Our main theorem is the following formulation of the Stable Adams Conjecture,
occurring in the text  as Theorem \ref{thm:stable-adams}.   

\vskip .1in

\begin{thm}
\label{thm:below}
Let ${\bf kU}$ denote the connective spectrum of (topological) complex K-theory.
If $p$ and  $\ell$ are distinct primes, then the maps 
$${\bf J}_\ell, \ {\bf J}_\ell \circ (\psi^p)_\ell:  ({\bf kU})_\ell \times_{{\bf K(\bZ_\ell,0)}} {\bf K(\bZ,0)} 
\quad \to \quad {\bf BG_\ell(S)}$$
are equal in the homotopy category of  connective $\Omega$-spectra.
\end{thm} 

\vskip .1in
One can view $ {\bf BG_\ell(S)}$
as  the even part of a connective delooping
of ${\bf GL_1(\mathbb S_\ell)}$, the units of the $\bZ/\ell$-completion of the sphere spectrum.
For a given prime $\ell$, we utilize the Bousfield-Kan $\bZ/\ell$-completion functor
\cite{B-K} on pointed simplicial sets which determines ``level-wise $\bZ/\ell$-completions" of spectra
arising from $\cF$-spaces as well as the Bousfield $\ell$-completion of spectra \cite{Bo}.   
Applied to the $\cF$-spaces we consider, the Bousfield-Kan 
$\bZ/\ell$-completion functor and the Bousfield $\ell$-completion functor yield homotopy equivalent 
spectra. (See Proposition \ref{prop:ell-spectra}.)

Theorem \ref{thm:below}  is a reformulation of a flawed statement of the fundamental result of \cite[Thm 10.4]{F80}
which fails to take into  into account the fact that the Frobenius map on ``algebraic spheres" 
does not preserve orientations.

Theorem \ref{thm:below} has the following oriented version, occurring in the text as 
Theorem \ref{thm:oriented-stable-adams}. 
We denote by $\bf{BU}$ the 1-connected cover of $\bf{kU}$ and by 
$\bf{BG^o(S)}_{[0]}$ the 1-connected cover of the 
delooping of the sphere spectrum.

\vskip .1in

\begin{thm}
\label{thm:below2}
Let ${\bf J}_\ell^o:  ({\bf BU})_\ell \ \to \ (\bf{BG^o(S)}_{[0]})_\ell$ denote the map 
of 1-connected $\Omega$-spectra induced by ${\bf J}_\ell$.
If $p$ and $\ell$ are distinct primes, then the following maps 
  $${\bf J}_\ell^o, \ {\bf J}_\ell^o \circ (\psi^p)_\ell:  ({\bf BU})_\ell 
\quad \to \quad (\bf{BG^o(S)}_{[0]})_\ell $$
are equal in the homotopy category of  connective $\Omega$-spectra.
 \end{thm}
 
 As for Theorem \ref{thm:below}, we prove Theorem \ref{thm:below2} by establishing a stronger, 
 unstable homotopy equivalence in the homotopy category of $\cF$-spaces.
 
 As a corollary of Theorem \ref{thm:below2}, we obtain
the following $\ell$-local version of the Stable Adams Conjecture, occurring as Corollary  \ref{cor:local}
in the text below.  We use the Bousfield-Kan $\bZ_{(\ell)}$-completion functor $(-)_{(\ell)}$ to
localize spaces at the prime $\ell$.
 
\begin{cor}
If $p$ and $\ell$ are distinct primes, the following maps of 1-connected spectra
$${\bf J}_{(\ell)}^o, \ {\bf J}_{(\ell)}^o \circ (\psi^p)_{(\ell)}:  ({\bf BU})_{(\ell)} 
\quad \to \quad (\bf{BG^o(S)}_{[0]})_{(\ell)} $$
are equal in the homotopy category of  connective $\Omega$-spectra.
\end{cor}

 \vskip .1in

 The category $\cF$ is the opposite category of Segal's category $\Gamma$; an 
 $\cF$-space is a (covariant) functor ${\ul \cB}: \cF \ \to \ (s.sets_*)$ from 
the category $\cF$ of finite pointed sets to the category 
$(s.sets_*)$ of pointed simplicial sets which sends ${\bf 0}$ to the singleton space $*$.
An $\cF$-space ${\ul \cB}: \cF \to (s.sets_*)$ is said to be special if
$ \ \times p_i: {\ul \cB}({\bf n}) \to \prod_{i=1}^n {\ul \cB}({\bf 1})$  is a weak homotopy 
equivalence for each $n > 0$, where $p_i$ is the map induced by the map ${\bf  n} \to {\bf 1}$ in $\cF$
 sending $j \in \{ 0,\ldots,n \}$ to $1$ if $j=i$ and to 0 otherwise.  If $\ul\cB$ is special, 
 then the natural map from $\ul\cB({\bf 1})$ to the 0-space of its associated $\Omega$-spectrum $||\ul\cB||^+$ is a 
 homotopy-theoretic group completion. (See \cite{Mc-Seg}.)  
  
 We use the $\cF$-space $Sin(\ul \cB GL(\bC)^{top})$ to model complex K-theory, 
we adopt D. Quillen's using of the geometric Frobenius \cite{Quillen68} and/or D. Sullivan's Teichm\"uller 
lifting of the arithmetic Frobenius \cite{Sul} to model the $\bZ/\ell$-completion $(\psi^p)_\ell$
of the Adams operator $\psi^p$ (see Theorem \ref{thm:Sul}), and we use the $\cF$-space $\ul \cB G(S^2)$ to
model the connective spectrum associated to homotopy self-equivalences of even dimensional spheres.   
We model the ${\bf J}$-map  of $\Omega$-spectra associated to the 
$J$-homeomorphism by the map of $\cF$-spaces $\cJ: Sin(\ul \cB GL(\bC)^{top}) \to Sin(\ul \cB G^{top}(\bC^+))$.
The concrete models of $\cF$-spaces  for objects in the homotopy category of spectra as  introduced by 
G. Segal \cite{Segal} fit well with input from algebraic geometry given by \'etale homotopy theory
 as introduced by M. Artin and B. Mazur \cite{A-M}.

Our approach is to produce a universal $\cF$-space $\ul\cB G(X)$ for
a given pointed, connected simplicial set $X$  and develop a theory
of equivalence classes of $X$-fibrations over $\ul\cB$ represented by $\ul\cB G(X)$ in order to investigate 
homotopy classes of maps from a ``general" $\cF$-space $\ul\cB$  to $\ul\cB G(X)$.  This
constitutes an $\cF$-space version  of the familar classification of fiber homotopy
equivalence classes of Hurwicz fibrations over a CW complex with a given fiber 
as in \cite{Stash}, \cite{May75}.
Theorem \ref{thm:X-universal} presents our simplest classification result for an arbitrary
pointed connected simplicial set $X$, using monoids of weak equivalences from $X^n$ to 
$Sin(|X^n|)$ (where $X^n$ is the $n$-th smash power of $X$) and the 
``sum pairing" on these monoids determined by smash products. 
To prove Theorem \ref{thm:below}, we establish in Theorem \ref{thm:X-ell-universal} the 
analog of Theorem \ref{thm:X-universal} for $\bZ/\ell$-complete $X$-fibrations built using 
monoids of mod-$\ell$-equivalences $X^n \to \bZ/\ell(X^n)$; to prove
Theorem \ref{thm:below2}, we utilize Corollary \ref{cor:X-fib-o-ell},
 the analog of Theorem \ref{thm:X-ell-universal} for oriented $\bZ/\ell$-complete $X$-fibrations.

Our strategy of proof is to use commutative diagrams of simplicial schemes
 to produce fiber homotopy equivalences between $\bZ/\ell$-completed $X$-fibrations.  To do 
 this, we utilize the functor $(-)^\wedge$ 
which sends a pointed simplicial scheme to a pointed simplicial set (see Definition \ref{defn:wedge}).
Key attributes of $X \mapsto X^\wedge$ are its functoriality, its formulation using the \'etale
site of $X$ (including the \'etale homotopy theory of Artin and Mazur), its use of 
$\bZ/\ell$-completion for some prime $\ell$ invertible in the structure
sheaf of $X$, and the fact that the natural map of simplicial sets $(X\times Y)^\wedge \to 
X^\wedge \times Y^\wedge$ is a homotopy equivalence for the simplicial schemes we consider.
 
 Although aspects of this paper resemble that of \cite{F80}, only traces of that 
 paper remain in this manuscipt for we have
 introduced numerous modifications and details in order to formalize that outline.  In Section 
 \ref{sec:alg-spheres}, we introduce a corrected definition of algebraic spheres
 over a field which behaves properly with respect to smash products.  In Section \ref{sec:rigid-etale},
 we provide a corrected definition of the rigid \'etale topological type of a simplicial
 scheme extending the definition in \cite{F82} for schemes.  In Section \ref{sec:aspects},
 we introduce the Quillen model category $(\cF{\text -}spec/\ul\cN)_*$ of pointed-connected
 special $\cF$-spaces $\ul\cB$ equipped with a split projection to the $\cF$-space 
 $\ul\cN$, where $\ul\cN({\bf n})$ is the  discrete set ${\bf N}^{\times n}$ with sum 
 pairing given by addition of $n$-tuples.  This structure of a special $\cF$-space provides compatible base points for components
 of $\ul\cB$.   We give various examples of $\cF$-spaces used in the proofs of the
 above results, examples formulated using a simplification of the construction introduced in \cite{Segal}.
   
  Our definition of $X$-fibration in Section \ref{sec:X-fib} is a corrected, simplified 
  version of earlier definitions.  We introduce the functor $X{\text -}fib(-)$
  on the homotopy category of special pointed-connected $\cF$-spaces which
  sends $\ul\cB$ to the set of fiber homotopy equivalence classes of $X$-fibrations.
  Theorem \ref{thm:X-universal} asserts that $\ul\cB G (X)$ represents this functor.
  Section \ref{sec:ell-complete} introduces $\bZ/\ell$-complete $X$-fibrations, 
   leading to the theory $\ul\cB \mapsto X_\ell{\text -}fib(\ul\cB)$ represented by $\ul\cB G_\ell(X)$.
  This is complemented by the formalization of oriented $X$-fibrations and oriented $\bZ/\ell$-complete
  $X$-fibrations in Section \ref{sec:or-ell-complete}.
 
  In Section \ref{sec:Adams}, we give a comparison of maps of $\cF$-spaces
  determined by the Adams operation $\psi^p$,
  the Teichm\"uller lift of the arithmetic Frobenius used by Sullivan in \cite{Sul}, and the 
  Frobenius map used by Quillen in \cite{Quillen}.

Throughout this paper, $k$ will denote an algebraically closed field of characteristic $p > 0$ and
$\ell$ will denote a prime different from $p$.  We caution the reader that we use 
$\bZ_\ell$ to denote the $\ell$-adic integers (equal to \ $\varprojlim_n \bZ/\ell^n$)
and we use $\bZ_{(\ell)}$ to denote the localization of $\bZ$ at the prime $\ell$.
We fix some $\ell$-th primitive root of unity in $k$, 
thereby identifying the \'etale sheaf $\mu_\ell$ on a $k$-variety with the constant sheaf $\bZ/\ell$. 
We frequently utilize  $(X)_\ell$ or $X_\ell$ to denote $(\bZ/\ell)_\infty(X)$, 
which we refer to as the $\bZ/\ell$-completion of $X$.    
When working with schemes over $\Spec R$ for $R$ equal to
$k$, \ $\bC$, \ or the Witt vectors $W(k)$ of $k$, we designate the fiber product over $\Spec R$ 
simply by $(-) \times (-)$.

We respectfully acknowledge the influence upon this work of innovative ideas introduced
by Graeme Segal in \cite{Segal} for the category of spectra, and the approaches of 
 Daniel Quillen in \cite{Quillen68}, \cite{Quillen} and Dennis Sullivan in \cite{Sul} in their the proofs the original 
Adams Conjecture.  We thank J. Peter May for his early guidance and  Nitya Kitchloo  
for reviving our interest in the Stable Adams Conjecture.  Finally, we thank anonymous 
referees for their constructive comments. 

\vskip .2in


\section{Frobenius maps and algebraic spheres}
\label{sec:alg-spheres}

We begin by recalling the (geometric) Frobenius endomorphism,
stated for simplicity for affine $k$-varieties over $k$ defined over $\bF_q$
but easily extended to simplicial varieties over $k$ defined over $\bF_q$.

\begin{prop}
\label{prop:Frob}
Let $V$ be an affine algebraic variety over $k$ with coordinate algebra $k[V]$,
a finitely generated commutative $k$-algebra.   Assume that $V$ is defined 
over some finite subfield $\bF_q \subset k$ with $q = p^d$, so that $k[V] = F_q[V_{\bF_q}] \otimes_{\bF_q} k$
for some finitely generated commutative $\bF_q$-algebra $\bF_q[V_{\bF_q})]$.    
\begin{enumerate}
\item
The $d$-th {\bf (geometric) Frobenius} $F^d: V \to V$ is map of $k$-varieties defined as
the base change to $k$ of the $q$-th power map
$V_{\bF_q} \to V_{\bF_q}$ given by sending $f \in \bF_q[V_{\bF_q}]$ to $f^q$.  
For $d = 1$, we use $F$ to denote $F^1$.
\item 
The  $d$-th {\bf (arithmetic) Frobenius} $\sigma^d: V \to V$ is the map of schemes (but
not of $k$-varieties) given
by $1 \otimes (-)^q: \bF_q[V_{\bF_q}] \otimes_{\bF_q}  k \ \to \ k[V_{\bF_q}] \otimes_{\bF_q}  k$.
\item
The {\bf total $q$-th power map} \ $(-)^q = F^d \circ \sigma^d = \sigma^d \circ F^d: V \ \to \ V$ 
(also called the $d$-th {\bf (absolute) Frobenius})
is the map of schemes sending each  $f \in k[V]$ to $f^q$.  This map  induces the identity map 
on the (rigid) \`etale homotopy type and \`etale cohomology of $V$.
 \end{enumerate}
 \end{prop}
 
 \vskip .1in

Smash products play an important role in our arguments.  The following definition formalizes our
smash products of ``algebraic spheres", correcting the definition in \cite[\S 8]{F80}.  To shorten
formulas we denote by $\Delta[1]^{\times 2}$ the product $\Delta[1]\times \Delta[1]$ and we denote
by $L$ the union $\Delta[1]\times 1 \cup_{(0,0)} 1\times \Delta[1]$ inside $\Delta[1]^{\times 2}$

\begin{defn}
\label{defn:smash-alg}
Let $R$ denote either $\bC$ (the complex numbers) or $k$ or the Witt vectors $W(k)$ of $k$ 
(a complete discrete valuation ring with residue field $k$ and field of fractions of characteristic 0).
Consider the simplicial mapping cone of the open embedding 
$\bA_R^{n}{\text -}\{ 0 \} \hookrightarrow \bA^n_R$,
\begin{equation}
\label{defn:S-alg}
S^{2n,alg}_R \ \equiv \ (\Spec R) \cup_{(\bA_R^n{\text-}\{ 0 \} \times 0)}  
(\bA_R^n{\text-}\{ 0 \} \times \Delta[1])  \cup_{(\bA_R^n{\text-}\{ 0 \} \times 1)}   \bA_R^n,
\end{equation}
a simplicial $R$-scheme.  Since $GL_n$ acts naturally on the diagram
$$\ast \ \leftarrow \ \bA^n{\text -}0  \stackrel{id \times 0}{\to} \ \bA^n{\text -}0\times \Delta[1]
\ \stackrel{id \times 1}{\leftarrow } \  \bA^n{\text -}0 \ \hookrightarrow \ \bA^n,$$
there is a natural $GL_{n,R}$ action on $S^{2n,alg}_R$. 

The smash product 
\begin{equation}
\label{eqn:define-wedge}
\wedge:  S^{2m,alg}_R \wedge S^{2n,alg}_R \quad \to \quad S^{2m+2n,alg}_R
\end{equation}
is the map of simplicial $R$-schemes which factors the composition of 
the natural map 
$$S^{2m,alg}_R \times S^{2n,alg}_R \ \to \ (\Spec R) \cup_{\bA_R^{m+n}{\text-}\{ 0 \} \times  L }
(\bA_R^{m+n}{\text-}\{ 0 \} \times \Delta[1]^{\times 2})  \cup_{\bA_R^{m+n}{\text-}\{ 0 \} \times (1,1) }   
\bA_R^{m+n}$$
induced by the embedding $(\bA_R^m{\text-}\{ 0 \} \times \Delta[1]) \times (\bA_R^n{\text-}\{ 0 \} \times \Delta[1])
\ \hookrightarrow \ \bA_R^{m+n}{\text-}\{ 0 \} \times \Delta[1]^{\times 2}$
and the map 
$$(\Spec R) \cup_{\bA_R^{m+n}{\text-}\{ 0 \} \times  L }
(\bA_R^{m+n}{\text-}\{ 0 \} \times \Delta[1]^{\times 2})  \cup_{\bA_R^{m+n}{\text-}\{ 0 \} \times (1,1) }   
\bA_R^{m+n} \ \to S^{2m+2n,alg}_R$$
 induced by the simplicial map $\Delta[1]^{\times 2} \to \Delta[1]$
which sends $L \subset \Delta[1]^{\times 2} $ to the vertex 0 and sends the 
vertex $(1,1) \in \Delta[1]^{\times 2} $ to the vertex 1.

We verify that 
\begin{equation}
\label{eqn:alg-equi}
\wedge: S^{2m,alg}_R \times S^{2n,alg}_R \ \to \ S^{2m+2n,alg}_R
\end{equation}
is $GL_{m,R} \times GL_{n,R}$-equivariant, by observing that $GL_{m,R} \times GL_{n,R}$
acts on the appropriate diagram.
\end{defn}

 \vskip .1in
 
 The following proposition establishes that the definitions of Definition \ref{defn:smash-alg} induce
 the expected maps on \'etale cohomology.  We remind the reader that the multiplicative group
 scheme $\bG_{m,R}$ is isomorphic as an $R$-scheme to $\bA^1{\text -} 0$ and the ``projective
 $n$-space" $\bP^n_R$ over $R$ is $Proj$ of the graded $R$-algebra $R[x_0,\ldots,x_n]$.

\begin{prop}
\label{prop:Pn}
For any positive integer $d$, the
the maps on \'etale cohomology induced by the (geometric) Frobenius map $F$
\begin{equation}
\label{eqn:H-Pn}
H_{et}^{2s}(\bP^n_k,\bZ/\ell^d)  \ \stackrel{F^*}{\to} \ \ H_{et}^{2n}(\bP^n_k,\bZ/\ell^d), \quad \quad 
H_{et}^{2s}(S^{2n,alg}_k,\bZ/\ell^d)  \ \stackrel{F^*}{\to} \ \ H_{et}^{2s}(S^{2n,alg}_k,\bZ/\ell^d)
\end{equation} 
are identified with multiplication by $p^s$ on $\bZ/\ell^d$. 
Moreover,
\begin{equation}
\label{eqn:H*-S}
H_{et}^i(S^{2n,alg}_k,\bZ/\ell^d) \simeq \bZ/\ell^d, \ i = 0, 2n, \quad H_{et}^i(S^{2n,alg}_k,\bZ/\ell^d) = 0, \ 
i \not=  0, 2n,
\end{equation}
and the smash product (\ref{eqn:define-wedge}) induces an isomorphism
\begin{equation}
\label{eqn:et-iso}
H^*_{et}(S^{2m+2n,alg}_k,\bZ/\ell^d) \quad \stackrel{\sim}{\to} \quad
H^*_{et}(S^{2m,alg}_k \wedge S^{2n,alg}_k,\bZ/\ell^d).
\end{equation}
\end{prop}

\begin{proof}
The assertion for $F^*$ on $H_{et}^{2s}(\bP^n_k,\bZ/\ell^d)$ follows from the functoriality of 
the cycle class map.  Using the ``geometric fibration" $\bA^1_k{\text-}0 \to \bA^n_k{\text-}0 \to \bP_k^{n-1}$,
we conclude that the action of $F^*$ on $H_{et}^{2n-1}( \bA^n_k{\text-}0,\bZ/\ell^d)$ is identified with multiplication
by $p^n$ on $\bZ/\ell^d$ and $H_{et}^{i}( \bA^n_k{\text-}0,\bZ/\ell^d) = 0, \ i \not= 0, 2n-1$.   

Computations involving the \'etale cohomology of the simplicial schemes $S^{2n,alg}_k$ are achieved by using the
spectral sequence $E_1^{s,t} = H^t_{et}(X_s,\bZ/\ell^d) \Rightarrow H^{s+t}_{et}(X_\bu,\bZ/\ell^d)$ for the etale 
cohomology of a simplicial scheme $X_\bu$ (see \cite[Prop 2.4]{F82}).  For the simplicial schemes occurring
in this proposition, the spectral sequences are readily compared with analogs occurring in topology.
\end{proof}

\vskip .1in

The next proposition addresses the issue of the compatibility of the 
smash product on (topological) spheres $S^{2n}$ with the smash product of spaces 
$|S_\bC^{2n,alg}|$ closely  related to $S^{2n,alg}_\bC$.   The most natural way
to introduce such compatibility is to view $S^{2n}$ as the one-point compactification 
$(\bC^n)^+$ of $\bC^n$ equipped with its analytic topology
and use the action of $GL_n(\bC)$ on $(\bC^n)^+$; in this case, we view 
the smash product as the map on one-point compactifications, 
$(\bC^m)^+ \times (\bC^n)^+ \ \to \ (\bC^{m+n})^+$,  
induced by $\bigoplus: \bC^m \times \bC^n \to \bC^{m+n}$.
Observe that the Lie group $GL(n,\bC)^{top}$ acts continuously on $(\bC^n)^+$.

\begin{prop}
\label{prop:map-cone}
We denote by $|S^{2n,alg}_\bC|$ the total space of simplicial topological space $(S^{2n,alg}_\bC)^{top}$ obtained
by applying the analytic topology functor to the simplicial scheme $S^{2n,alg}_\bC$
over $\bC$.  The Lie group $GL(n,\bC)^{top}$ acts naturally on $|S^{2n,alg}_\bC|$.  

Consider the continuous pointed map 
 \begin{equation}
 \label{eqn:tau-n}
 \gamma_n: |S^{2n,alg}_\bC| \quad \to \quad (\bC^n)^+
 \end{equation}
 given by sending $\ast \in |S^{2n,alg}_\bC|$ to the base point $\infty \in (\bC^n)^+$;
 sending $(z,t) \in \bA^n{\text-}0 \times [0,\frac{1}{2}]$ to $\frac{z}{4t} \in (\bC^n)^+$;
 sending $(z,t) \in \bA^n{\text-}0 \times [\frac{1}{2},1]$ to $tz$, and sending 
 $\bA^n$ (attached at $t=1$) via the inclusion in $(\bC^n)^+$.
So defined, $\gamma_n$ is $\GL_n^{top}$-equivariant as well as a pointed homotopy equivalence.

Then smash product determines the following commutative square whose maps are
$GL_m^{top} \times GL_n^{top}$-equivariant
\begin{equation}
\label{eqn:GL-smash}
\xymatrix{
|S^{2m,alg}_\bC| \times |S^{2n,alg}_\bC| \ar[d]_{\gamma_m \times\gamma_n} \ar[r]^-\wedge & |S^{2m+2n,alg}_\bC| \ar[d]^{\gamma_{m+n}} \\
(\bC^m)^+ \times (\bC^n)^+ \ar[r]^\wedge & (\bC^{m+n})^+
.}
\end{equation}
\end{prop}

\begin{proof}
We view $|S^{2n}_\bC|$ is the mapping cone of the inclusion
$(A^n(\bC){\text-}\{ 0 \})^{top} \hookrightarrow (A^n(\bC))^{top}$.
We see by inspection that $\gamma_n$ is $GL_n^{top}$-equivariant.
 One way to see that $\gamma_n$ is a homotopy equivalence
 is by viewing $S^{2n}$ as homotopy equivalent to the the mapping cone of 
 $S^{2n-1} \hookrightarrow \bR^{2n}$.

The commutativity of (\ref{eqn:GL-smash}) is verified by inspection.
\end{proof}

\vskip .1in

We shall use the ``simplicial bar construction" for topological groups and algebraic groups.  For
the reader's convenience, we specify our notation for these simplicial objects.

\begin{notation}
\label{note:bar}
For a  closed subgroup scheme $\bG_R$ of $GL_{N,R}$ for some 
$N$ (we shall take $R$ to be $\bC$ or 
$W(k)$ or $k$), we denote by $B\bG_R$ the simplicial $R$-scheme obtained by
applying the ``simplicial bar construction"; so defined, $B\bG_R$ has $\Spec R$ in simplicial 
degree 0 and $\bG_R^{\times n}$ in simplicial degree $n > 0$; 
face maps $\bG_R^{\times n} \to \bG_R^{\times n-1}$ are determined by multiplications of adjacent 
copies of $\bG_R$ as well as projections; degeneracy 
maps $\bG_R^{\times n} \to \bG_R^{\times n+1}$ involve projections 
and the identity map $\Spec R \to \bG_R$.   If $X_\bu$ is a simplicial $R$-scheme
with $\bG_R$ acting on the left on each $X_n$ and with each structure map of $X_n \to X_m$ a map of 
$\bG_R$-schemes, then we define $B(\bG_R,X_\bu)$ to be the diagonal of the evident 
bisimplicial $R$-scheme:   $B(\bG_R,X_\bu)_0 = X_0$, \ $B(\bG_R,X_\bu)_n \ = \ 
\bG_k^{\times n}\times_{\Spec R} X_n$, and with face maps determined by multiplication 
for $\bG_r$ except that the last face map entails the action $\bG_R \times X_n \to X_{n-1}$.
More generally, if also given $Y_\bu$  simplicial $R$-scheme
with $\bG_R$ acting on the right on each $Y_n$, we consider $B(Y_\bu,\bG_R,X_\bu)$,
 the diagonal of the evident tri-simplicial $R$-scheme.

In order to use a parallel construction for Lie groups $G^{top}$ acting continuously 
on a simplicial topological 
space $T$, we use the simplicial bar constructions to define $BG^{top}$ and $B(G^{top},T)$ as the
total spaces of the simplicial spaces $(BG^{top})_\bu$ and $B(G^{top},T)_\bu$; here 
$(BG^{top})_0$ is a point and $(BG^{top})_n$ the topological space $(G^{top})^{\times n}$ for $n > 0$;
$B(G^{top},T)_0$ is the space $T_0$ and $B(G^{top},T)_n$ is the space $(G^{top})^{\times n} \times T_n$; the face and
degeneracy maps are defined as above.

In what follows, we shall suppress the designation $(-)_\bu$ for simplicial schemes and
simplicial spaces.
\end{notation}

\vskip .1in

The relevance of Proposition \ref{prop:map-cone} to our investigation of sphere fibrations
begins to appear in the following observation.

\begin{prop}
\label{prop:geom-fiber}
Consider the commutative square of simplicial $k$-varieties
\begin{equation}
\label{eqn:tau-F}
\xymatrix{
B(GL_{n,k},S_k^{2n,alg}) \ar[d]_{\tau_n^{alg}} \ar[r]^F  & B(GL_{n,k},S_k^{2n,alg}) \ar[d]^{\tau_n^{alg}} \\
BGL_{n,k} \ar[r]_F & BGL_{n,k}
}
\end{equation}
whose horizontal maps are the (geometric) Frobenius map and whose vertical maps are the natural projections.
Then the restriction of \ $F: B(GL_{,k},S_k^{2n,alg}) \to B(GL_{n,k},S_k^{2n,alg})$ above
the base point $\Spec k \to BGL_{n,k}$ is the Frobenius map $F: S_k^{2n,alg} \to S_k^{2n,alg}$.
\end{prop}

\vskip .2in


\section{$\bZ/\ell$-completions and rigid \'etale topological types}
\label{sec:rigid-etale}

We begin by specifying the $\bZ/\ell$-completions  we shall use.  Throughout,
$\ell$ will denote a prime number.

\begin{defn}
\label{defn:ell-complete}
We employ the $\bZ/\ell$-completion functor of \cite[X.4.9,X.4.10]{Bo-Kan}
$$(\bZ/\ell_\infty)(-): (s.sets_*) \ \quad \to \quad (s.sets_*)$$
(taking values in Kan complexes)
which has the important property that $(\bZ/\ell)_\infty(f): (\bZ/\ell)_\infty(X) \to 
(\bZ/\ell)_\infty(Y)$ is a homotopy equivalence whenever the map $f:X \to Y$
of simplicial sets is a mod-$\ell$ equivalence (i.e., whenever $f$
 induces an isomorphism $H_*(X,\bZ/\ell) \ \stackrel{\sim}{\to} \ H_*(Y,\bZ/\ell)$).
Another useful property of $(\bZ/\ell_\infty)(-)$ is that the natural map $(\bZ/\ell_\infty)(X\times Y) 
\to (\bZ/\ell_\infty)(X) \times (\bZ/\ell_\infty)(Y)$ is a homotopy equivalence with
a natural left inverse \cite[I.7.2]{Bo-Kan}.

To simplify notation, we shall often abbreviate the name of the functor 
$(\bZ/\ell_\infty)(-)$ by setting
$$(-)_\ell \ \equiv \ (\bZ/\ell_\infty)(-): (s.sets_*) \quad \to \quad (s.sets_*).$$ 
\end{defn}

\vskip .1in

Using the properties of $(\bZ/\ell_\infty)(-)$ mentioned in Definition \ref{defn:ell-complete},  
we easily verify that the smash product $S^{2m} \times S^{2n} \ \to \ S^{2m+2n}$ of spheres
(for example, represented by the map of one-point compactifications $(\bC^m)^+ \times (\bC^n)^+
\to (\bC^{m+n})^+$ induces a smash product natural with respect to $m$ and $n$ 
\begin{equation}
\label{eqn:wedge}
(Sin(S^{2m}))_\ell \times (Sin(S^{2n}))_\ell \quad \to \quad (Sin(S^{2m}\times S^{2n}))_\ell
\quad \to \quad Sin((S^{2m+2n}))_\ell.
\end{equation}

\vskip .1in

 We recall that a 
rigid \'etale cover $U \to X$ of a scheme $X$ consists of a disjoint union $\coprod_{\ul x \in \ol X} U_{\ul x}$,
where $U_{\ul x}$ is a connected scheme \'etale over $X$ equipped with a geometric point $\Spec k(\ul x) \to 
U_{\ul x}$ mapping to some $\ul x: \Spec k(\ul x) \to X$ in $\ol X$, where the indexing set $\ol X$ is the
 ``set of geometric points $\ol X$ " of $X$ as discussed in  \cite[\S 4]{F82}. We denote by $\cC(X,r{{\text - }}et)$ 
 the category (which is a partially ordered set) of rigid \'etale coverings $U \to X$.  This is functorial in $X$ in the sense that if 
$ Y \to X$ is a map of (geometrically pointed) schemes, then 
$$\cC(X,r{{\text - }}et) \ \to \ \cC(Y,r{{\text - }}et), \quad \coprod_{\ul x \in \ol X} U_{\ul x} \ \mapsto \ 
\coprod_{\ul y \in \ol Y} Y\times_X U_{f(\ul y)}.$$

\vskip .1in

\begin{defn}
\label{defn:pro-type}
As discussed in \cite[\S 4, \S 8]{F82}, the ($\check{C}$ech) \'etale topological type functor on the 
category of locally noetherian schemes
$$(-)_{ret}: (\text{geometrically pointed schemes}) \quad \to \quad (pro{\text-}s.sets_*)$$
is defined by setting 
$$X_{ret}: (\cC(X,r{{\text - }}et))^{op} \ \to \ (s.sets_*), \quad (U \to X) \ \mapsto \ \pi_0(N_X(U)),$$ 
sending a rigid \'etale cover $U \to X$ to 
the simplicial set of connected components of its $\check{C}$ech nerve.
\end{defn}

\vskip .1in

To obtain a functor from geometrically pointed simplicial schemes to pointed simplicial sets we use the
Bousfield-Kan homotopy (inverse) limit functor (see \cite[XI.3.2]{Bo-Kan}) 
$$\underset{\longleftarrow}{holim}(-): (pro-s.sets_*) \ \to \ (s.sets_*).$$ 
functorial with respect to maps of pro-simplicial sets determined by a functorial map of indexing categories.
 The following definition differs somewhat from that of \cite{F82} for a simplicial scheme $n \mapsto X_n$
 in that it utilizes $(X_n)_{ret}$ for each $n$.
 \vskip .1in
 
\begin{defn}
\label{defn:wedge}
For a geometrically pointed, connected simplicial scheme $X_\bu$ and a prime $\ell$, we define the $\bZ/\ell$-completed
rigid \'etale topological type functor 
$$(-)^\wedge: (s.schemes_*)  \quad \to \quad (s.sets_*)$$
by sending the geometrically pointed simplicial scheme $X_\bu$ to
\begin{equation}
\label{eqn:hat}
(X_\bu)^\wedge \quad \equiv \quad (diag\{ n \mapsto \underset{\longleftarrow}{holim}((X_n)_{ret})_\ell) \})_\ell,
\end{equation}
where $diag\{ - \}: (bis.sets_*) \to (s.sets_*)$ denotes the diagonal functor sending a bi-simplicial set to its diagonal.  
(See \cite[App B]{Bo-F}.)

If $X_\bu$ is not connected, then we require a geometric pointing of each connected component  of $X_\bu$
and then take the disjoint union of  the result of applying $(-)^\wedge$ to each component.
\end{defn}

\vskip .1in

\begin{remark}
\label{rem:explain}

In the above definition, we apply $(-)_\ell$ after taking the diagonal of a bisimplicial set
because the diagonal of a simplicial object of Kan complexes is not necessarily a 
Kan complex.   

For the schemes considered in this paper, the natural map
$(X_n)_{ret} \to ((X_n)_{ret})_\ell$ induces an isomorphism on $\bZ/\ell$-cohomology
(i.e., $(X_n)_{ret}$ is $\ell$-good) and  $\pi_1((X_n)_{ret})$ acts trivially on $H^*((X_n)_{ret},\bZ/\ell)
= H^*_{et}(X_n,\bZ/\ell)$.
Because $\ol X_n$ has bounded cardinality and the homotopy groups of are finitely generated 
$\bZ_\ell$-modules, this implies that $\underset{\longleftarrow}{holim}((X_n)_{ret})_\ell)$ is also 
$\ell$-good. 

Consequently, the $\bZ/\ell$-cohomology of $diag\{ n \mapsto \underset{\longleftarrow}{holim}((X_n)_{ret})_\ell) \}$
equals the \'etale cohomology $H^*_{et}(X_n,\bZ/\ell)$.  Thus, applying $(-)_\ell$ after taking the 
diagonal of bisimplicial set does not change the $\bZ/\ell$-cohomology and yields a Kan complex.
\end{remark}

\vskip .1in

The input to the following theorem is the ``classical 
comparison theorem" relating the mod-$\ell$ \'etale cohomology of a complex algebraic
variety $V$ (denoted by $H^*_{et}(V,\bZ/\ell)$ to  cohomology of the analytic 
space of complex points $V(\bC)$ (denoted by $H^*(V(\bC)^{top},\bZ/\ell)$)
established by M.  Artin and A. Grothendieck \cite{SGA4}, subsequently 
sharpened to pro-finite homotopy types by Artin and B. Mazur in \cite{A-M}.
(See also \cite{Deligne}.)

	For a pointed simplicial scheme $X$ of finite type over $\bC$, we denote by
$Sin(X(\bC)^{top})$   the diagonal of the bi-simplical set 
$(s,t) \mapsto  Sin_t(X_s(\bC)^{top})$, where 
$Sin(X_s(\bC)^{top})$ is the singular complex of 
the topological space of complex points of $X_s$ equipped with the 
analytic topology.

\vskip .1in

\begin{thm} 
\label{thm:holim} 
Let $X_\bC$ be a pointed, connected simplicial scheme which is of finite type over $\bC$
in each simplicial degree.  Then there
exist homotopy equivalences
\begin{equation}
\label{eqn:Phi_X} 
Sin(X(\bC)^{top})_\ell \
\quad \stackrel{\rho_{U\mapsto X}}{\leftarrow} \quad \underset{\longleftarrow}{holim} (Sin((N_X(U)_\bu(\bC))^{top})_\ell
\quad \stackrel{\rho_{top\mapsto \pi_0}}{\to} \quad (X_\bC)^\wedge
\end{equation}
natural with respect to pointed maps of simplicial schemes over $\bC$.
\end{thm}

\begin{proof}
The map $\rho_{U\mapsto X}$ is determined level-wise maps for each $n \geq 0$
 by applying $Sin(|N_{X_n}(-)^{top}|)$ to rigid  \'etale coverings $U \to X_n$,
whereas the map $\rho_{top\mapsto \pi_0}$ is determined level-wise by the ``inverse system" of maps \\
$Sin(N_{X_n}(U)^{top}) \to \pi_0(N_{X_n}(U)^{top})$.
The proof that these maps  are homotopy equivalences can be found in \cite[Thm 8.4, Cor 8.5]{F82}.
\end{proof}

\vskip .1in

Fix an embedding $W(k) \hookrightarrow \bC$, where $W(k)$ denotes the Witt
vectors of $k$.  This embedding, together the
quotient map $W(k) \to k$, determines distinguished geometric points $\Spec k 
\to \Spec W(k)$, \ $\Spec \bC \to W(k)$.
From the point of view of the \'etale topology, $\Spec W(k)$ is contractible.
For many smooth (simplicial) schemes $X_R$ over $R$, the base change maps
$X_\bC \to X_{W(k)}, \ X_k \to X_{W(k)}$ induce isomorphisms in \'etale cohomology
\begin{equation}
\label{eqn:base-change}
H^*_{et}(X_\bC,\bZ/\ell^d) \  \stackrel{\sim}{\leftarrow} \ H^*_{et}(X_{W(k)},\bZ/\ell^d) \ 
 \  \stackrel{\sim}{\to} \ H^*_{et}(X_k,\bZ/\ell^d).
 \end{equation}
 See \cite{Deligne} and \cite{Milne}.
 For example, any scheme $X_\bZ$ which is proper and smooth over $\Spec \bZ$
 enables such base changes isomorphisms, as 
 well as the complement in such a proper, smooth scheme of a divisor with normal
 crossings; another important example for us is a reductive group scheme $\bG_{\bZ}$ over
 $\Spec \bZ$ and the bar construction $B\bG_{\bZ}$ applied to such a reductive
 group scheme.   See, for example, \cite{F82}, \cite{F-Par}.

\vskip .1in

As a consequence of the base changes isomorphisms of (\ref{eqn:base-change}),
one obtains homotopy equivalences of completed \'etale homotopy types as
established in \cite[\S 8]{F82}.  The formulation and proof of Proposition 
\ref{prop:comparison} is parallel to that for
Theorem \ref{thm:holim}. To arrange that the maps of (\ref{eqn:natural}) preserve 
base points, we require that the simplical scheme $X_{W(k)}$ be connected
and the existence of a section $\Spec W(k) \to (X_{W(k)})_0$ relating  base points.
(We point out that if 
$Z= \Spec W(k), \ V = Z-\Spec k$, then $Z\coprod V \to Z$ is the unique rigid \'etale cover
of $Z$ and $\pi_0(N_Z(Z\coprod V))$ is contractible, so that we may naturally identify 
the two pointings of $(X_{W(k)})^\wedge$ determined by the section $\Spec W(k) \to (X_{W(k)})_0$.)

\begin{prop}
\label{prop:comparison} 
Equip $W(k)$ with an embedding into $\bC$ and consider a connected simplicial scheme $X_{W(k)}$ 
over $W(k)$ with a section $\Spec W(k) \to (X_{W(k)})_0$ such $X_{W(k)}$ is of  finite type over $W(k)$
in each simplicial degree.  Assume that $X_{W(k)}$
satisfies the \'etale cohomological base change isomorphisms of (\ref{eqn:base-change}).
Then base change determines pointed homotopy equivalences
\begin{equation}
\label{eqn:natural} 
(X_\bC)^\wedge  \quad \stackrel{\rho_{W \mapsto \bC}}{\leftarrow} \quad
 (X_{W(k)})^\wedge  \quad \stackrel{\rho_{W \mapsto k}}{\rightarrow}  \quad (X_k)^\wedge.
 \end{equation}
 
 These equivalences are natural with respect to maps $X_{W(k)} \to Y_{W(k)}$ 
defined over $\Spec W(k)$.
\end{prop}

\vskip .1in

\begin{remark}
\label{rem:alt-def}
Proposition \ref{prop:comparison} provides motivation/justification for our definition of
$(-)^\wedge$.  In \cite{F82}, \'etale coverings (and hypercoverings) of simplicial schemes
were considered, rather than considering ``level-wise" \'etale coverings.  This works well if 
one considers pro-objects in the homotopy category of pointed simplicial sets but the
rigidification process seems questionable for such coverings.

Another approach is due to G. Quick in \cite{Quick} and \cite{Quick2} whose ``\'etale topological type"
of a simplicial scheme is a simplicial object in a suitable homotopy category for 
pro-spaces.  We find this approach elegant, and an interested reader might find
Quick's formulation to be preferable to that given in Definition \ref{defn:wedge} though
perhaps requiring additional machinery in the context of the homotopy category of pro-spaces.
\end{remark}

\vskip .1in

\begin{ex}
\label{ex:S-alg}
$S^{n,alg}_{W(k)}$ satisfies the hypotheses of Proposition \ref{prop:comparison}.
This can be verified by considering base changes at each level of the simplicial schemes.  See the
proof of Proposition \ref{prop:Pn}.
\end{ex}

\vskip .1in

One immediate consequence of Theorem \ref{thm:holim} and Proposition \ref{prop:comparison} 
is the following assertion concerning 
the behavior of $(-)^\wedge$ with respect
to products of simplicial varieties.   This property enables us to conclude that applying
$(-)^\wedge$ to certain $\cF$-objects of simplicial schemes yields $\cF$-spaces which are special.

\begin{cor}
\label{cor:product}
Consider pointed, connected
simplicial schemes $X_{W(k)}, \ Y_{W(k)}$ over $\Spec W(k)$ of finite type in each dimension which satisfy the
base change equivalences of (\ref{eqn:natural}).  Assume that $X \times Y$ satisfies
the condition that the natural map 
$$H^*(X(\bC)^{top},\bZ/\ell) \otimes H^*(Y(\bC)^{top},\bZ/\ell)  \quad \to \quad H^*(X(\bC)^{top} \times Y(\bC)^{top},\bZ/\ell)$$
is an isomorphism.  Then  the natural map 
$$(X_R \times  Y_R)^\wedge \ \to \ (X_R)^\wedge \times (Y_R)^\wedge$$
is a pointed homotopy equivalence whenever $R$ equals $k, \ \bC,$ or $W(k)$.   
\end{cor}

\vskip .1in

We interpret the following proposition as asserting that the Frobenius map $F$ induces
a self-map of the $\bZ/\ell$-completed sphere fibration $(\tau^{alg}_n)^\wedge$ over $(BGL_{n,k})^\wedge$
which restricts to a map of degree $p^n$ on homotopy fibers.

\begin{prop}
\label{prop:hom-fiber}
There is a natural homotopy equivalence from $ (S_k^{2n,alg})^\wedge$ to the 
homotopy fiber of the map obtained by applying $(-)^\wedge$ to the projection map $\tau_{n,k}^{alg}$:
$$(\tau_{n,k})^\wedge: (B(GL_{n,k},S_k^{2n,alg}))^\wedge \ \to \ (BGL_{n,k})^\wedge.$$ 

The commutative squares of  (\ref{eqn:tau-F}) of simplicial $k$-varieties determines
$$F^\wedge: (\tau_{n,k})^\wedge \quad \to \ \quad F^{\wedge*}((\tau_{n,k})^\wedge)$$
 whose map on homotopy fibers is homotopic to 
$F^\wedge: (S_k^{2n,alg})^\wedge \ \to \ (S_k^{2n,alg})^\wedge$
which has degree $p^n$.
\end{prop}

\begin{proof}
By Proposition \ref{prop:geom-fiber}, $S_k^{2n,alg}$ is the geometric fibre of $\tau_{n,k}$.
By \cite[Thm 10.7]{F82}, there is a natural $\bZ/\ell$-equivalence from 
$(S_k^{2n,alg})^\wedge$ (the ``homotopy
type of the geometric fiber") to the homotopy fiber of $(\tau_{n,k})^\wedge$ (the 
``homotopy-theoretic fiber").
Using Example \ref{ex:S-alg} and Proposition \ref{prop:comparison}, we conclude that the second 
assertion follows from the naturality of the comparison of geometric and homotopy-theoretic
fibers.
\end{proof}

\vskip .2in


\section{Some aspects of $\cF$-spaces}
\label{sec:aspects}

We require compatible structures on the collection of $(S_k^{2n,alg})^\wedge$-fibrations 
$(\tau_{n,k}^{alg})^\wedge$ of Proposition \ref{prop:hom-fiber} 
in order to specify a homotopy class of spectra from the spectrum associated to
 $\langle \{  (BGL_{n,k})^\wedge, n\geq 0 \} , \ \bigoplus\rangle$
to the spectrum associated to \\
$\langle\{  BG_\ell(S^{2n}),n\geq 0 \}, \ \wedge \rangle$.  The structure we employ is
that of $\cF$-spaces.

We remind the reader that the category $\cF$ (opposite to Segal's category $\Gamma$)
has objects consisting of the finite pointed sets ${\bf n} = \{ 0,1,\dots,n\}$ and maps 
${\bf m} \to {\bf n} $ which are pointed (i.e., 0-preserving) maps of sets.  A functor from
$\cF$ to a category $\cC$ equipped with a given final-cofinal object which sends
 ${\bf 0}$ to this final-cofinal object is called an $\cF$-object of $\cC$.
Of particular interest is the case in which $\cC$ is the category of pointed simplicial sets; 
in this case, we call such a functor an $\cF$-space. 

Recall that an $\cF$-space $\ul\cB$ is said to be special if $\times p_i: \ul \cB({\bf n}) \to
\prod_{i=1}^n \ul\cB({\bf 1})$ is a weak homotopy equivalence for each $n$.  For a special
$\cF$-space $\ul\cB$, the map $\mu: \ul\cB({\bf 2}) \to \ul\cB({\bf 1})$ is a ``sum" pairing 
which is commutative up to all higher homotopies, where $\mu: {\bf 2} \to  {\bf 1} \in \cF$
sends both $1,2 \in {\bf 2}$ to $1 \in {\bf 1}$.

The category of special $\cF$-spaces (which we denote by $\cF[s.sets_*]$) is given 
 in \cite[Thm 3.5]{Bo-F} the structure of a proper closed simplicial model category.
 In this model, a map of special $\cF$-spaces $f: \ul\cB^\prime \to \ul\cB$ is a weak 
equivalence provided that $f({\bf n}): \ul\cB^\prime({\bf n}) \to \ul\cB({\bf n})$ is a weak 
equivalence for all ${\bf n} \in \cF$ (i.e., is a ``level-wise weak equivalence"); we 
denote the associated homotopy category by $Ho(\cF[s.sets_*])$.
We refer to \cite{Bo-F}, \cite{F80} for discussions
of model-theoretic aspects of  $\cF[s.sets_*]$
such as fibrant replacement of objects/maps and cofibrant replacement
of objects in the category.  

\vskip .1in

\begin{defn}
\label{defn:assoc-spec}
Let $T$ be the finite simplicial set with two non-degenerate simplices and with
$|T| = S^1$.  Let $T^n$ denote the $n$-th smash power of $T$.  For any 
$\cF$-space $\ul\cB$, we extend $\ul\cB$ so that it is a functor from 
finite pointed simplicial sets to simplical sets.  

Following \cite[Prop 1.4]{Segal} and \cite[Thm 4.4]{Bo-F}, we define the spectrum
$$|| \ul\cB|| \quad \equiv \quad \{ \ul\cB(T^n), \ T \wedge \ul\cB(T^n) \to \ul\cB(T^{n+1})\}$$
associated to $\ul\cB$. 
\end{defn}

\vskip .1in

The $\cF$-space $\ul\cN$ mentioned in the introduction is our first example.

\begin{ex}
\label{ex:cN}
We denote by $\ul \cN: \cF \to (sets_*)$  the functor sending ${\bf n}$ for $n > 0$ to the discrete
pointed simplicial set $\bN^{\times n}$ (the pointed set of $n$-tuples of non-negative integers)
and sending $\alpha: {\bf n} \to {\bf m}$ in $\cF$ to the pointed set map
$\bN^{\times n} \to \bN^{\times m}$ which sends the $n$-tuple $I=(i_1,\dots,i_n)$ to the $m$-tuple
$\alpha(I) = (j_1,\dots,j_m)$ where $j_t = \sum_{\alpha(s)= t} i_s$ and the sum is the
monoid structure of $\bN$.

The spectrum $||\ul\cN ||$ has $\pi_0$ equal to $\bN$.  The associated $\Omega$-spectrum
 $||\ul\cN ||^+$ (see (\ref{eqn:equiv})) can be identified with ${\bf K(\bZ,0)}$.
\end{ex}

\vskip .1in

We summarize results of \cite{Bo-F} describing the relationship between the homotopy 
category of $\cF$-spaces and the homotopy category of connected spectra.  An $\cF$-space
$\ul\cB$ is said to be very special if it is special and if $\pi_0(\ul\cB({\bf 1}))$ is a group.

\begin{thm} \cite[\S 5]{Bo-F}
\label{thm:summary}
Denote by $vs\cF[s.sets_*]$  the full subcategory of $\cF[s.sets_*]$
whose objects are very special $\cF$-spaces.
There is a functor 
 \ $T: \cF[s.sets_*] \to vs\cF[s.sets_*]$ \ together with a 
natural transformation $\eta: id \to T$ which induces a map on homotopy categories
fitting in the following commutative square
\begin{equation}
\label{eqn:equiv}
\xymatrix{
Ho(\cF[s.sets_*]) \ar[d]_T \ar[r] ^{||-||} & Ho(c.spectra) \ar[d]^+ \\
Ho(vs\cF[s.sets_*])  \ar[r] & Ho(c.\Omega\text{-}spectra)
.}
\end{equation}
\begin{enumerate}
\item
The lower horizontal map of (\ref{eqn:equiv}) is an equivalence of categories, where
$Ho(c.\Omega{\text -}spectra)$ is the ``usual" stable homotopy category of connective spectra.
\item
For any $\cF$-space $: \ul\cB$, $\eta: \ul\cB \to T(\ul\cB)$
is a stable weak equivalence in the sense of \cite{Bo-F}.
\item
If $\ul\cB$ is a special $\cF$-space, then $\eta: ||\ul\cB|| \to ||T(\ul\cB)||$ is a ``homotopy-theoretic group
completion" on $0$-spaces  $||\ul\cB||_0 \to ||T(\ul\cB)||_0$ and determines homotopy equivalences
$||\ul\cB||_i \to ||T(\ul\cB)||_i$ for $i > 0$.
\item
If $\ul\cB$ is a very special $\cF$-space, then $\eta: ||\ul\cB|| \to ||T(\ul\cB)||$ 
is a homotopy equivalence of $\Omega$-spectra.
\end{enumerate}
\end{thm}

\vskip .1in

We view the functor $T: Ho(\cF[s.sets_*]) \ \to \ Ho(vs\cF[s.sets_*])$ as a localization
with respect to stable equivalences,  namely those maps $\ul\cB \to \ul\cB^\prime$ in 
$\cF[s.sets_*]$ whose induced map on the stable homotopy groups 
$\pi_*(\ul\cB) \to \pi_*(\ul\cB^\prime)$ is an isomorphism.

\vskip .1in

\begin{construct}
\label{construct:homotopy}
Denote by $\cF{\text -}spec$ the full subcategory of $\cF[s.sets_*]$ of special 
$\cF$-spaces.  Then $\cF{\text -}spec$ inherits the structure of a Quillen
model category from that of $\cF[s.sets_*]$ such that
$$Ho(\cF{\text -}spec) \quad  \to \quad Ho(\cF[s.sets_*])$$
 is fully faithful.  Namely, one verifies 
that any fibrant replacement and any cofibrant replacement of an object of 
$\cF{\text -}spec$ is a map in $\cF{\text -}spec$.

Moreover, the slice category $\cF{\text -}spec/\ul\cN$ inherits the structure of a Quillen
model category from that of $\cF{\text -}spec$ such that
$$Ho(\cF{\text -}spec/\ul\cN) \quad  \to \quad Ho(\cF[s.sets_*])$$
is also fully faithful.

Denote by $(\cF{\text -}spec/\ul\cN)_*$ the full subacategory of $\cF{\text -}spec/\ul\cN$
whose objects 
$$f: (\ul\cB \to \ul\cN, \ s: \ul\cN \to \ul\cB;\  f\circ s = id)$$ 
are special $\cF$-spaces over $\ul\cN$  equipped with a section
with the additional property the pre-image
$\ul\cB_I$ of any $I \in \ul\cN({\bf n})$ is pointed, connected.  Then 
$(\cF{\text -}spec/\ul\cN)_*$ inherits the structure of a Quillen
model category from that of $\cF{\text -}spec/\ul\cN$ such that
$$Ho((\cF{\text -}spec/\ul\cN)_*) \quad  \to \quad Ho(\cF{\text -}spec/\ul\cN)$$
is fully faithful.  The justification for this is that fibrant and cofibrant replacements
in $\cF{\text -}spec/\ul\cN$ of objects in $(\cF{\text -}spec/\ul\cN)_*$ can
be achieved using functorial base-point preserving constructions 
level-by-level (see, for example, \cite{Bo-F}), and the base point of the connected 
space $\ul\cB_I$ for 
$(\ul\cB,f,s)$ for each $I \in \ul\cN({\bf n})$ is given by $s(I)$.
\end{construct}

\vskip .1in

\begin{cor}
\label{cor:homotopy-cat}
Two maps $\phi, \psi: {\ul \cB} \to {\ul \cB}^\prime$ in $(\cF{\text -}spec/\ul\cN)_*$
are equal in \\
$Hom_{\cF[s.sets_*]}(\ul\cB, \ul\cB^\prime)$ if and only if they are equal in 
$Hom_{(\cF{\text -}spec/\ul\cN)_*}(\ul\cB, \ul\cB^\prime)$.   

Moreover, if $\phi, \psi$ are equal in $Hom_{(\cF{\text -}spec/\ul\cN)_*}(\ul\cB, \ul\cB^\prime)$,
then they are homotopic in the stable homotopy category of 
connective spectra.
\end{cor}

\vskip .1in

\begin{remark}
The constructions in this paper involve objects and maps in $(\cF{\text -}spec/\ul\cN)_*$.
The additional structure for an $\cF$-space $\ul\cB \to \ul\cN$ to be an object of $(\cF{\text -}spec/\ul\cN)_*$
enables internal structure maps to be base-point preserving, needed when one takes smash products.

We refer to an object of $(\cF{\text -}spec/\ul\cN)_*$ as a ``pointed-connected special $\cF$-space $\ul\cB$ over $\ul\cN$",
leaving the structure map $f: \ul\cB \to \ul\cN$ and the section $s: \ul\cN \to \ul\cB$ implicit.
\end{remark}

\vskip .1in

As seen in \cite[Cor 2.2]{Segal} (see also \cite{May}), topological permutative categories determine $\cF$-spaces
over $\ul\cN$.  

By replacing the category of finite pointed sets by the full subcategory in which ${\bf n}$ is the unique
set in its isomorphism class,  we simplify Segal's construction for \cite[Cor 2.2]{Segal} for a category $\cC$ 
with sums.  Segal considers categories $\cC(S)$ for each finite pointed set $S$
whose objects are functors from the category $\cP(S)$ of subsets of $S$ with maps being inclusions
to $\cC$ which take disjoint unions to sums.  If we  replace $\cP(S)$ by the full subcategory of subsets of $\cP({\bf n})$
(i.e., choose a unique representative ${\bf n}$ for a finite set $S$ with a base point and $n$ non-base-point elements),
then Segal's construction simplifies to 
\begin{equation}
\label{eqn:simplify-X}
{\ul \cB}G(X)({\bf n})_I \quad \equiv \quad BG(X^I) \times (\prod_{S\subset {\bf n}} E\Sigma_{|S|}),
\end{equation}
where $|S|$ is the number of non-base point elements of $S$.  

The permutations groups $\Sigma_{|S|}$
play an essential role in defining $\alpha: {\ul \cB}G(X)({\bf n})_I \ \to \ {\ul \cB}G(X)({\bf n})_{\alpha(I)}$.
Following \cite{Segal}, the justification of this model comes from envisioning the $n$-th simplicial category as having
object ``space" the discrete set of pairs $(I \in \ul\cN({\bf n}), \sigma_s \in \prod_{S\subset {\bf n}} \Sigma_{|S|}$)
and having endomorphism``spaces" of such pairs equal to $G(X^I) \times \prod_{S\subset {\bf n}} \Sigma_{|S|}$.

We give two such ``topological examples" relevant to our main theorems. 

\vskip .1in

\begin{ex} \cite[Ex 8.1]{F80}, \cite{Segal}
\label{ex:BGL}
There exists a special $\cF$-object over $\ul\cN$ of pointed topological spaces 
$${\ul \cB}GL(\bC)^{top}: \cF \quad \to \quad (\text{spaces}_*)$$
with ${\ul \cB}GL(\bC)^{top}({\bf 1})\ = \ \coprod_{n\geq 0} BGL_n(\bC)^{top}$ 
and sum pairing given by block sum of matrices.  Here,
$BGL_n(\bC)^{top}$ is the total topological space associated to the simplicial bar construction 
applied to the complex Lie group $GL_n(\bC)^{top}$ which we view as the nerve of the 
topological category with one object $\bC^n$ and whose mapping space is $GL_n(\bC)^{top}$.
For $n > 1$ and $I \in \ul\cN$, we identify 
${\ul \cB}GL(\bC)^{top}({\bf n})_I$ with
\begin{equation}
\label{eqn:id-GL}
{\ul \cB}GL(\bC)^{top}({\bf n})_I \quad \simeq \quad
\prod_{i=1}^n  BGL_{n_i}(\bC)^{top} \times \prod_{S\subset {\bf n}} E(\Sigma_{|S|},\Sigma_{|S|}).
\end{equation}
This is  the nerve of the topological category $\cC(GL)({\bf n})$ an object of which is indexed by pairs
$(I = (i_1,\ldots,i_n) \in \cF({\bf n}), \times_{S\subset {\bf n}} \sigma_S \in \Sigma_{|S|})$ where $S\subset {\bf n}$ 
is the set of subsets of $\{1,\ldots,n\}$  of size $|S| \geq 2$.  Maps of $\cC(GL)({\bf n})$
consist of isomorphisms of objects, where $Iso((I \in \cF({\bf n}), \times_{S\subset {\bf n}} \sigma_S \in \Sigma_{|S|})$
can be identified with $\prod_{j=1}^n GL_{i_j}(\bC)^{top} \times \prod_{S\subset {\bf n}} \Sigma_{|S|}$.
Sending ${\bf n}$ to $\cC(GL)({\bf n})$ determines a well-defined
functor from $\cF$ to (topological categories), where elements $\mu_S \in \Sigma_{|S|}$ are used to 
provide functorial maps between partial sums indexed by $S \subset {\bf n}$. 
Taking the nerve of these categories, one obtains the ``special" $\cF$-object 
${\ul \cB}GL(\bC)^{top}: \cF \to \ (\text{spaces}_*)$ over $\ul\cN$.

Moreover, there is a map of special $\cF$-objects over $\ul\cN$ of
topological spaces
\begin{equation}
\label{eqn:tau-top-spec}
\tau^{top}_{\bC^+}: {\ul \cB}(GL(\bC),\bC^+)^{top} \to {\ul \cB}GL(\bC)^{top}: \ \cF \quad \to \quad (\text{spaces}_*)
\end{equation}
with the property that 
$${\ul \cB}(GL(\bC),\bC^+)^{top}({\bf 1})\ = \ \coprod_{n\geq 0} B(GL_n(\bC),(\bC^n)^+)^{top} \ \to$$
$$\coprod_{n\geq 0} BGL_n(\bC)^{top} \ = \ {\ul \cB}GL(\bC)^{top}({\bf 1})$$
is the projection.  

The construction of ${\ul \cB}(GL(\bC),\bC^+)^{top} $ is given by taking the nerves of 
 the topological categories $\widetilde{\cC(GL)}({\bf n})$, an object of which is indexed by pairs
$$(z \in \wedge_{j=1}^n (\bC^{i_j})^+,\ \times_{S\subset {\bf n}} \sigma_S \in \Sigma_{|S|}).$$
Again, the indexing set $S\subset {\bf n}$ is the set of subsets of $\{1,\ldots,n\}$  of size $|S| \geq 2$.
Maps of $\widetilde{\cC(GL)({\bf n})}$ consist of isomorphisms of objects which respect sums; these 
can be identified with $\prod_{j=1}^n GL_{i_j}(\bC)^{top}$.
\end{ex}

\vskip .1in

\begin{ex}
\label{ex:tau}
Proposition \ref{prop:map-cone} enables us to replace $(\bC^n)^+$ by $|S^{2n,alg}_\bC|$
in the construction of $\tau^{top}_{\bC^+}$, thereby obtaining
\begin{equation}
\label{eqn:tau-S2alg-top}
\tau_{|S^{2,alg}_\bC|}^{top}: {\ul \cB}(GL(\bC),|S^{2,alg}_\bC|)^{top} \to {\ul \cB}GL(\bC)^{top}: 
\ \cF \quad \to \quad (\text{spaces}_*)
\end{equation}
\end{ex}

\vskip .1in
\begin{ex}
\label{ex:G(|X|)}
Let $X$ denote a pointed, connected, locally finite simplicial set with realization $|X|$.  Denote by 
$G(|X|)^{top}$ the function space with the compact open topology of pointed self-equivalences
of $|X|$.    Composition gives $G(|X|)^{top}$ the structure of an $H$-space.
Moreover, smash product $X^m \times X^n \to X^{m+n}$ determines  a continuous
smash product $|X^m|\times |X^n| \to |X^{m+n}|$ leading to the smash product 
$G(|X^m|)^{top} \times G(|X^n|)^{top} \to G(|X^{m+n}|)^{top}$ of self-equivalences and a continuous action of 
$G(|X^m|)^{top} \times |X^m| \to |X^m|$ commuting with this smash product.

There exists a special $\cF$-object over $\ul\cN$ of pointed topological spaces 
$${\ul \cB}G(|X|)^{top}: \cF \quad \to \quad (\text{spaces}_*)$$
with ${\ul \cB}G(|X|)^{top}({\bf 1})\ = \ \coprod_{n\geq 0} BG(|X^n|)^{top}$ 
and sum pairing given by smash products.   As in the previous example,
$BG(|X^n|)^{top}$ is the total topological space associated to the simplicial bar construction 
applied to the $H$-space $G(|X^n|)^{top}$.  For $n > 1$ and $I = (i_1,\ldots,i_n)\in \ul\cN$, we identify 
${\ul \cB}G(|X|)^{top}({\bf n})_I$ with
\begin{equation}
\label{eqn:id-GL-I}
{\ul \cB} G(|X|)^{top}({\bf n})_I \quad \simeq \quad
\prod_{j=1}^n  BG(|X^{i_j}|)^{top} \times \prod_{S\subset {\bf n}} E(\Sigma_{|S|},\Sigma_{|S|}).
\end{equation}

As in the previous example, there there is a map of special $\cF$-objects over $\ul\cN$ of
topological spaces
\begin{equation}
\label{eqn:tau-top}
\pi^{top}_{|X|}: {\ul \cB}(G(|X|),|X))^{top} \to {\ul \cB}G(|X|)^{top}: \ \cF \quad \to \quad (\text{spaces}_*)
\end{equation}
with the property that 
$${\ul \cB}(G(|X|),|X|)^{top}({\bf 1})\ = \ \coprod_{n\geq 0} B(G(|X^n)|,|X^n|)^{top} \ \to$$
$$\coprod_{n\geq 0} BG(|X^n|)^{top}\ = \ {\ul \cB}G(|X|)^{top}({\bf 1})$$
is the projection.  
\end{ex}

\vskip .1in

By applying $Sin(-)$ to the previous examples, one obtains $\cF$-spaces.  We proceed to 
construct $\cF$-spaces within the context of pointed simplicial sets, motivated by these 
topological examples.

\vskip .1in

We remind the reader that $Sin(T)$ is a Kan complex for any topological space.  
To avoid the awkwardness that smash products of Kan complexes need not be
Kan complexes, we consider  monoids of weak equivalences $X^i \to Sin(|X^i|)$
rather than monoids of self-equivalences of Kan complexes.

\vskip .1in

\begin{defn}
\label{defn:G(X)}
Let $X$ be a pointed, connected simplicial set with geometric realization $|X|$.  Recall 
that $|-|$ is left adjoint to the singular functor $Sin(-): (spaces_*) \to (s.sets_*)$.
We consider the 
function complex $\ul{Hom}_*(X,Sin(|X|))$ whose $t$-simplices are simplicial maps
$X \times \Delta[t] \to Sin(|X|)$ subject to the condition that $x_0 \times \Delta[t]$ maps to 
$x_0 \subset Sin_0(|X|)$ where $x_0 \in X_0$ is the base point.  The adjunction $| - | \circ Sin(-) \to id$ equips 
$\ul{Hom}_*(X,Sin(|X|))$ with a simplicial monoid structure associated with composition
of functions and defines the natural action
\begin{equation}
\label{eqn:nat-act}
\ul{Hom}_*(X,Sin(|X|)) \times Sin(|X|) \quad \to \quad Sin(|X|).
\end{equation}

We denote by $G(X) \ \subset \ \ul{Hom}_*(X,Sin(|X|))$ 
the (simplicial) submonoid consisting of components whose images in 
the discrete monoid $\pi_0 (\ul{Hom}_*(X,Sin(|X|)))$ are invertible, and we
denote by  $G^o(X) \subset G(X)$ the distinguished component containing 
the canonical map $X \to Sin(|X|)$.  

If the action of each $\sigma \in \Sigma_n$ on smash 
factors of $X^n$ is weakly homotopy equivalent to the identity of $X^n$, 
then we say that $\Sigma_n \subset G^o(X^n)$.

The adjunction equivalence \ $\ul{Hom}_*(X,Sin(|X|)) \ \stackrel{\sim}{\to} \ Sin(\ul{Hom}^{top}_*(|X|,|X|)$
implies that $G(X)$ is $G(X)$ is a Kan complex.  If $X$ is locally finite, then $G(X) \ \simeq Sin(G(|X|)^{top})$.
\end{defn}

\vskip .1in

\begin{ex}
\label{ex:sphere-spectrum}
If $|X|$ has the homotopy type of the (pointed) $n$-sphere $S^n$, 
then $G(X)$ is homotopy equivalent to $Sin(\Omega_{\pm 1}^n(S^n))$
and $G^0(X)$ is homotopy equivalent to $Sin(\Omega_1^n(S^n))$.  Thus,
$\pi_i(G(X)) = \pi_{i+n}(S^n)$ for $i > 0$ and $\pi_0(G(X)) = \bZ/2$.  
See, for example, a discussion by J.F. Adams in \cite[\S2]{Adams}.
\end{ex}

\vskip .1in

The following special example of a map of $\cF$-spaces
$$\pi_X:  \ul\cB (G(X),Sin(|X|))  \quad \to \quad \ul\cB G(X),$$
is the simplicial analog of Example \ref{ex:G(|X|)}.  This will serve as the ``universal $X$-fibration".

\begin{ex}
\label{ex:G(X)-act}
Let $X$ be a pointed, connected simplicial set.
For $I = (i_1,\ldots,i_n) \in \ul\cN({\bf n})$, denote by $X^I$ the product $\prod_{j=1}^n X^{i_j}$, 
where $X^{i_j}$ is the smash product of $i_j$ copies of $X$.  We
consider the simplicial monoid $G(X^I) \ \subset \ \ul{Hom}_*(X^I,Sin(|X^I|))$ with product 
given by composition (together with the adjunction map $|-| \circ Sin \to id$).
\begin{enumerate}
\item
Let $\alpha: {\bf n} \to {\bf m}$ be a non-decreasing map in $\cF$, $I \in \ul\cN({\bf n})$, 
and $J = \alpha(I) \in \ul\cN({\bf m})$.  Then $\alpha$ determines maps involving smash products
\begin{equation}
\label{eqn:smash-prod}
X^I \ \stackrel{\alpha}{\to} \  X^J, \quad G(X^I) \ \stackrel{\alpha}{\to} \  G(X^J).
\end{equation}
\item
For $\alpha: {\bf n} \to {\bf m}$ non-decreasing, the natural actions of the form (\ref{eqn:nat-act})
fit into commutative squares 
\begin{equation}
\label{eqn:compat}
\xymatrix{
G(X^I) \times X^I \ar[r] \ar[d]_\alpha &  Sin(|X^I |) \ar[d]^\alpha \\
G(X^J) \times X^J  \ar[r] &  Sin(|X^J |).
}
\end{equation}
\end{enumerate}

The compatible pairings of (\ref{eqn:compat}) determine the map of special $\cF$-spaces over $\ul\cN$
\begin{equation}
\label{eqn:pi-X}
\pi_X:  \ul\cB (G(X),Sin(|X|))  \quad \to \quad \ul\cB G(X)
\end{equation}
whose value on ${\bf 1}$ is the projection $\coprod_{i \geq 0} B(G(X^i),Sin(|X^i|)) \ \to \ \coprod_{i \geq 0} BG(X^i).$
As in previous examples, we can identify the restriction of $\pi_X$ above $I \in \cN({\bf n})$ as the projection 
$$
\prod_{i=1}^n  BG(X^n),Sin(|X^n)) \times \prod_{S\subset {\bf n}} E(\Sigma_{|S|},\Sigma_{|S|}) \ \to \ 
\prod_{i=1}^n  BG(X^n) \times \prod_{S\subset {\bf n}} E(\Sigma_{|S|},\Sigma_{|S|}).
$$

If $\Sigma_n \subset G^o(X^n)$ for all $n > 0$, then we define the ``oriented analogue" 
 $\ul \cB G^o(X) \hookrightarrow \ Sin(\ul\cB G(|X|)^{top,o})$
by simply replacing $G(X^I)$ by $G^o(X^I)$  in the construction of $\ul\cB G(X)$.
\end{ex}

\vskip .1in

The map  $\pi^{top}_{|X|}$ of Example \ref{ex:G(|X|)} and  and the map $\pi_X$ of Examle \ref{ex:G(X)-act}
are closely related, as we next state.

\begin{prop}
\label{prop:close-relate}
Let $X$ denote a pointed, connected, locally finite simplicial set with realization $|X|$. 
Then there is a commutative square of pointed-connected special $\bF$-spaces over $\ul\cN$
\begin{equation}
\label{eqn:compat-relate}
\xymatrix{
\ul\cB (G(X),Sin(|X|))  \ar[r]  \ar[d]_{\pi_X} &  \ar[d]^{\pi^{top}_{|X|} } Sin(\ul\cB (G(|X|),|X|)^{top}) \\
 \ul\cB G(X) \ar[r]_{adj} & Sin(\ul\cB G(|X|)^{top})
}
\end{equation}
whose horizontal maps arise using the canonical maps $G(X^I) \to Sin(G(|X^I|))$.
For every $I \in \ul\cN({\bf n}), n > 0$, the restriction of (\ref{eqn:compat-relate}) above $I$
\begin{equation}
\label{eqn:compat-I}
\xymatrix{
\ul\cB (G(X),Sin(|X|))_I  \ar[r]  \ar[d]_{\pi_X} &  \ar[d]^{\pi^{top}_{X} } Sin(\ul\cB (G(|X|),|X|)^{top})_I \\
 \ul\cB G(X)_I  \ar[r]_{adj}  & Sin(\ul\cB G(|X|))_I
}
\end{equation}
is a cartesian square whose vertical maps are fibrations with fibers homotopy
equivalent to $Sin(|X^I |)$ and whose horizontal maps are weak equivalences.
\end{prop}

\vskip .2in


\section{$X$-fibrations over $\cF$-spaces}
\label{sec:X-fib}

For  a pointed, connected simplical set $X$ and $I = (i_1,\ldots,i_n) \in \ul\cN({\bf n})$,
$X^I$ denotes $\prod_{j=1}^n X^{i_j}$.  If $\alpha: {\bf n} \to {\bf m}$ is a non-decreasing map which 
sends $I \in \ul\cN({\bf n})$ to $J \in \ul\cN({\bf m})$,
then $\alpha$ determines maps $\wedge_{\alpha,I}: X^I \to X^J$ and $G(X^I) \to G(X^J)$.  Since the smash product
is not commutative, these maps are not functorial with respect to all maps of $\cF$.  To obtain functoriality, one
must incorporate actions of symmetric groups on the factors.

If $\ul\cB$ is a pointed-connected special $\cF$-space over $\ul\cN$, then an $X$-fibration $f: \ul\cE \to \ul\cB$ 
is an $\cF$-space structure over $\ul\cB$ which packages the (left) actions of $G(X^I)$ on 
fibers of $f_I: \ul\cE_I \to \ul\cB_I$ compatible with smash products.  The role of these actions becomes 
more explicit when one replaces $f: \ul\cE \to \ul\cB$ by the homotopy equivalent $X$-fibration
$\pi_{P(f),X}: \ul\cB(P(f),G(X),X) \to \ul\cB(P(f),G(X))$ of Proposition \ref{prop:2-sided}

\begin{defn}
\label{defn:X-fibration}
Consider  a pointed, connected simplicial set $X$ and a pointed-connected special $\cF$-space $\ul\cB$
over $\ul\cN$.  An $X$-fibration over $\ul\cB$ is a pair of maps 
\ $(f: \ul\cE \ \to \ \ul\cB, \ s_f: \ul\cB \to \ul\cE)$ \ of special $\cF$-spaces over $\ul \cN$
with $f \circ s_f = id$ satisfying the following conditions:
\begin{enumerate}
\item
For each $n > 0$ and $I =(i_1,\ldots,i_n) \in \ul\cN({\bf n})$, $f_I: \ul\cE_I\to \ul\cB_I $ is a fibration with fibers which
are pointed homotopy equivalent to $Sin(|X^I|)$.
\item
For any non-decreasing $\alpha: {\bf n} \to {\bf  m} \in \cF$,  $I \in \ul\cN({\bf n})$,  $J = \alpha(I) \in \ul\cN({\bf m})$,
and $\Delta[d] \to \ul\cB_I$, the restriction of $\alpha$ above $\Delta[d]$, 
$$\alpha_{I,\Delta[d]}: \ul\cE_I \times_{\ul\cB_I} \Delta[d] \ \to \ \ul\cE_J \times_{\ul\cB_J} \Delta[d],$$
is homotopic to $(\wedge_{\alpha,I})\times id: Sin(|X^I |)\times \Delta[d] \quad \to \quad Sin(|X^J|) \times \Delta[d].$
\end{enumerate}
\end{defn}

Observe that if $\theta: {\bf n} \to {\bf n} \in \cF$ is a pointed bijection, $I \in \ul\cN({\bf n})$ and $J= \theta(I)$, 
and $\Delta[d] \to \ul\cB_I$ is a $d$-simplex,
then $\theta_{I,\Delta[d]}: \ul\cE_I \times_{\ul\cB_I} \Delta[d] \ \to \ \ul\cE_J \times_{\ul\cB_J} \Delta[d]$ is necessarily
homotopic to  $\theta \times id: Sin(|X^I |)\times \Delta[d] \quad \to \quad Sin(|X^J|) \times \Delta[d]$
as can be seen by utilizing  Definition \ref{defn:X-fibration}(1) for each projection $\pi: {\bf n} \to {\bf 1}$
and the condition that $f: \ul\cE \to \ul\cB$ is a map of special $\cF$-spaces.

\vskip .1in

\begin{remark}
\label{rem:Reedy}
Definition \ref{defn:X-fibration} is a simplified formulation of the definition of an $X$-fibration given in
\cite{F80} and corrected in \cite{B-K}.  We have replaced the somewhat cumbersome
condition that an $X$-fibration $f:\ul\cE \to \ul\cB$ be ``properly sectioned" 
by the condition that $f$ is equipped with a left inverse.  This enables compatible 
base points on fibers of $f$.
\end{remark}

\vskip .1in

\begin{defn}
\label{defn:X-map}
Let $X$ be a pointed, connected simplicial set.
Consider a map $\phi: \ul\cB \ \to \ \ul \cB^\prime$ of pointed-connected special $\cF$-spaces over
$\ul\cN$, and $X$-fibrations $(f: \ul\cE \to \ul \cB, \ s_f: \ul\cB \to \ul\cE)$ \ 
$(f^\prime: \ul\cE ^\prime\to \ul \cB^\prime, \ s_{f^\prime}: \ul\cB^\prime \to \ul\cE^\prime)$.
A map $(\tilde \phi,\phi): (f,s_f) \ \to \ (f^\prime,s_{f^\prime})$  of $X$-fibrations over $\phi$ is  
a commutative square  in $(\cF{\text -}spec/\ul\cN)_*$.
\begin{equation}
\label{eqn:F-map}
\xymatrix{
\ul\cE \ar[r]^{\tilde \phi}  \ar[d]_f  &  \ul\cE^\prime \ar[d]^{f^\prime} \\
\ul\cB \ar[r]_{\phi }& \ul\cB^\prime
}
\end{equation}
compatible with sections for $f$ and $f^\prime$
with the property that the induced map ${\ul\cE}_I  \ \to \ {\ul\cE^\prime}_I \times_{\ul\cB_I^\prime} \ul\cB_I$
is a fiber homtopy equivalence over $\ul\cB_I$ for all $I \in \ul\cN({\bf n})$.
\end{defn}

\vskip .1in

The main result of this section is Theorem \ref{thm:X-universal} telling us that the functor
sending a pointed-connected special $\cF$-space $\ul\cB$ over $\ul\cN$ to the set of  equivalence
classes of $X$-fibrations over $\ul\cB$ (as defined in Definition \ref{defn:Xfib}) is representable
in $Ho(\cF[s.sets_*]/\ul\cN)$ by $\ul\cB G(X)$ of Example \ref{ex:G(X)-act}  with the map 
$\pi_X:  \ul\cB (G(X),Sin(|X|))  \ \to \ \ul\cB G(X)$ of (\ref{eqn:pi-X}) serving as the universal $X$-fibration.   

\vskip .1in

\begin{ex}
\label{ex:X-fib}
We give a few examples of $X$-fibrations.
\begin{enumerate}
\item
Let $X$ be a pointed, connected simplicial set.
The map of $\cF$-spaces 
$$\pi_X: \ul\cB (G(X),Sin(|X|)) \quad \to \quad \ul\cB G(X)$$
 of Example \ref{ex:G(X)-act}
is a $X$-fibration over $\ul\cB G(X)$.
\item
If $X$ is a pointed, connected simplicial set such that $\Sigma_n \subset G^o(X^n)$ for all $n > 0$, 
then $\pi^o_X$ restricts to the $X$-fibration
$$\pi^o_X: \ul\cB (G^o(X),Sin(|X|)) \quad \to \quad \ul\cB G^o(X).$$
\item
Applying $Sin(-)$ to the map $\tau^{top}_{\bC^+}:{\ul \cB}(GL(\bC),\bC^+)^{top} \to {\ul \cB}GL(\bC)^{top}$ of (\ref{eqn:tau-top})
determines the $X$-fibration
\begin{equation}
\label{eqn:tau-C+}
\tau_{\bC^+}: Sin({\ul \cB}(GL(\bC),\bC^+)^{top}) \quad \to \quad  Sin({\ul \cB}GL(\bC)^{top})
\end{equation}
 for some pointed simplicial set $X$ with $|X| \simeq \bC^+$.
\item
Similarly, if $T_n(\bC) \subset GL_n(\bC)$ denotes the diagonal torus, we obtain
\begin{equation}
\label{eqn:tau-T-C+}
\tau_{T,\bC^+}: Sin({\ul \cB} (T(\bC),\bC^+)^{top}) \quad \to \quad  Sin({\ul \cB}T(\bC)^{top})
\end{equation}
\end{enumerate}
\end{ex}

\vskip .1in

\begin{lemma}
\label{lem:iota!-construct}
Consider a pointed, connected simplicial set $X$ and a pointed-connected special $\cF$-space $\ul\cB$ over $\ul\cN$,
and an $X$-fibration  $(f:\ul\cE \to \ul\cB, \ s_f: \ul\cB \to \ul\cE$).  Let $\iota: \ul\cB \to \ul\cB^\prime$ be a
trivial cofibration of pointed, connected special $\cF$-spaces over $\ul\cN$ and 
factor $\iota \circ f$ as $\iota_!(f) \circ j: \ul\cE  \ \to \ \ul\cE^! \to \ul\cB^\prime$ with $j$ a 
trivial cofibration and  $\iota_!(f)$ a fibration.
Then $\iota_!(f): \ul\cE^! \to \ul\cB^\prime$ is an $X$-fibration over $\ul\cB^\prime$ (with section 
$s_{f^!}: \ul\cB^\prime \to \ul\cE^!$ obtained by extending $\iota \circ s_f$).
\end{lemma}

\begin{proof}
Consider 
$\iota^*(\iota_!(f)):  \ul\cE^! \times_{\ul\cB^\prime} \ul\cB \ \to \ \ul\cB$ and 
some $n > 0$ and $I \in \cF({\bf n})$.  Observe that
 $\ul\cE^! \times_{\ul\cB^\prime} \ul\cB \ \to \ \ul\cE^!$ is a weak equivalence,
the pull-back of the weak equivalence $\iota$ along $(\iota_!(f))_I$.  (This uses the fact that 
the category of $(s.sets_*)$ is right proper.  See \cite{nLab}.)  Thus, the map of fibrations
$f_I \to (\iota^*(\iota_!(f)))_I$ is a map of fibrations over $\ul\cB_I$ whose map on total spaces is 
a weak equivalence.
\end{proof}

For notational simplicity, we suppress explicit choices of sections $s_f$.

\begin{defn}
\label{defn:Xfib}
Consider a pointed, connected simplicial set $X$ and a pointed-connected special $\cF$-space $\ul\cB$ over $\ul\cN$.
We define an equivalence relation on $X$-fibrations over $\ul\cB$ by setting  $f \ \sim \ f^\prime$
if and only there is a chain of weak equivalences of pointed-connected special $\cF$-spaces over $\ul\cN$
\begin{equation}
\label{eqn:self-chain}
\ul\cB \ \stackrel{p}{\twoheadleftarrow} \quad \ul\cB_1 \quad \stackrel{\phi}{\to} \quad \ul\cB_2  
\quad \stackrel{j}{\hookleftarrow}  \ \quad \ul\cB
\end{equation}
and a  map of $X$-fibrations $p^*(\ul\cE) \to j_!(\ul\cE^\prime)$ covering $\phi: \ul\cB_1 \to \ul\cB_2$
(with sections)
$$p^*(\ul\cE) \quad \to \quad \iota_!(\ul\cE^\prime),$$
where $p$ is a cofibrant replacement of $\ul\cB$, $j$ is a fibrant replacement of $\ul\cB^\prime$,  and the 
image of (\ref{eqn:self-chain}) in $Hom_{Ho((\cF{\text -}spec/\ul\cN)_*)}(\ul\cB,\ul\cB)$ is the identity.

We denote by $X{\text-}fib(\ul\cB)$ the set of equivalence classes of $X$-fibrations over $\ul\cB$,
which we view as ``fiber homotopy equivalence classes" of $X$-fibrations over $\ul\cB$.
\end{defn}

\begin{remark}
\label{rem:restate}
Two $X$-fibrations $f , \ f^\prime$ over $\ul\cB$ are equivalent if and only if 
there is a chain of weak equivalences in $(\cF{\text -}spec/\ul\cN)_*$
\begin{equation}
\label{eqn:chain-equiv}
\ul\cB \ \stackrel{\phi_1}{\to} \ \ul\cB_1 \ \stackrel{\phi_2}{\leftarrow} \ \ul\cB_2 \to \cdots \to  \ul\cB_{2n-1}
\stackrel{\phi_{2n}}{\leftarrow} \ul\cB
\end{equation}
whose image in $Hom_{Ho((\cF{\text -}spec/\ul\cN)_*)}(\ul\cB,\ul\cB)$ is the identity and which is 
covered by a chain of maps of $X$-fibrations $f \to f^1 \leftarrow f^2 \to \cdots \to f^{2n} \leftarrow f^\prime$.
\end{remark}

\vskip .1in

The following two lemmas  enable a proof of functoriality for $\ul\cB \ \mapsto \ X{\text-}fib(\ul\cB)$.
The first concerns behavior of $X$-fibrations with respect to trivial cofibrations $\iota: \ul\cB \to \ul\cB^\prime$ 
and the second provides the analogous behavior with respect to trivial fibrations $p: \ul\cB^{\prime\prime} \to \ul\cB$.

\begin{lemma}
\label{lem:iota!}
Consider a pointed, connected simplicial set $X$ and a pointed-connected special $\cF$-space $\ul\cB$ over $\ul\cN$.
If $\iota: \ul\cB \to \ul\cB^\prime$ is a trivial cofibration, then sending an $X$-fibration $f: \ul\cE \to \ul\cB$ to the
$X$-fibration $\iota_!*(f): \ul\cE^! \to \ul\cB^\prime$ determines a bijection 
\begin{equation}
\label{eqn:iota}
\iota_!: X{\text-}fib(\ul\cB) \quad \to \quad X{\text-}fib(\ul\cB^\prime).
\end{equation} 

If $f: \ul\cE \to \ul\cB$ is an $X$-fibration over $\ul \cB$,
there is a natural map $f  \ \to \ \iota^*(\iota_!(f))$ of $X$-fibrations over $\ul\cB$,
and if $f^\prime: \ul\cE^\prime \to \ul\cB^\prime$ is an $X$-fibration over $\ul \cB^\prime$,
there is a  natural map $\iota_!(\iota^*(f^\prime)) \ \to \ f^\prime$ of $X$-fibrations 
over $\ul\cB^\prime$.
\end{lemma}

\begin{proof}
To show that $\iota$ is well defined on equivalence classes of $X$-fibrations over $\ul\cB$, 
it suffices to observe that $\iota$ applied to (\ref{eqn:chain-equiv}) determines a commutative
diagram
\begin{equation}
\label{eqn:i!-map}
\xymatrix{
\ul\cB \ar[d]_\iota  & \ar[l]_{\stackrel{p}{\twoheadleftarrow}} \ar[d] \ul\cB_1 \ar[r]^\phi \ &  \ul\cB_2  \ar[d] &
 \ul\cB  \ar[d]^\iota \ar[l]_{\stackrel{j}{\hookleftarrow}} \\
\ul\cB^\prime  & \ar[l]^{\stackrel{p^\prime}{\twoheadleftarrow}}  \ul\cB^\prime_1 \ar[r]_{\phi^\prime}  &  \ul\cB^\prime_2   & \ul\cB^\prime  
\ar[l]^{\stackrel{j^\prime}{\hookleftarrow}}
}
\end{equation}
where the left square is chosen to be a pull-back square, the right square is a pushout square, and $\phi^\prime$ is 
chosen to make the middle square commute. 

Consider $\iota^*(\iota_!(f)):  \ul\cE^! \times_{\ul\cB^\prime} \ul\cB \ \to \ \ul\cB$ and 
some $n > 0$ and $I \in \cF({\bf n})$.  Observe that
 $\ul\cE^! \times_{\ul\cB^\prime} \ul\cB \ \to \ \ul\cE^!$ is a weak equivalence,
the pull-back of the weak equivalence $\iota$ along $(\iota_!(f))_I$.  (This uses the fact that 
the category of $(s.sets_*)$ is right proper.  See \cite{nLab}.)  Thus, the map of fibrations
$f_I \to (\iota^*(\iota_!(f)))_I$ is a map of fibrations over $\ul\cB_I$ whose map on total spaces is 
a weak equivalence.  Using the long exact sequence for homotopy groups of a fibration,
we conclude that $f_I \to (\iota^*(\iota_!(f)))_I$ induces homotopy equivalences on fibers
so that $f \to \iota^*(\iota_!(f))$ is a map of $X$-fibrations over $\ul\cB$.  
 
 On the other hand, given an $X$-fibration $f^\prime: \ul\cE^\prime \to \ul\cB^\prime$ over $\ul\cB^\prime$,
 the fibration of $\iota_!(\iota^*(f^\prime)): (\iota^*(\ul\cE^\prime)^! \to \ul\cB^\prime$ factors through
 $f^\prime$ with  $\iota_!(\iota^*(f^\prime)) \to f^\prime$ a map of $X$-fibrations over $\ul\cB^\prime$.
\end{proof}

We omit the straight-forward proof of the next lemma which is analogous to that of Lemma \ref{lem:iota!}.

\begin{lemma}
\label{lem:p*}
Consider a pointed, connected simplicial set $X$, a pointed-connected special $\cF$-space $\ul\cB$ over $\ul\cN$, 
and a trivial fibration $p: \ul\cB^{\prime\prime} \to \ul\cB$. Then
\begin{equation}
\label{eqn:p*}
p^*:  X{\text-}fib(\ul\cB) \quad \to \quad X{\text-}fib(\ul\cB^{\prime\prime})
\end{equation}
is well defined bijection.

Moreover, if $f^{\prime\prime}  : \ul\cE^{\prime\prime}  \ \to \ \ul\cB^{\prime\prime} $ is an $X$-fibration 
over $\ul\cB^{\prime\prime} $, then the composition $p\circ f^{\prime\prime} : \ul\cE^{\prime\prime} \ \to \ \ul\cB$
is an $X$-fibration over $\ul\cB$ with the property that the natural map $p^*(p\circ f^{\prime\prime}) \to f^{\prime\prime}$
is a map of $X$-fibration over $\ul\cB^{\prime\prime}$.  Furthermore, if $f: \ul\cE \to \ul\cB$ is an $X$-fibration
over $\ul\cB$, then $p \circ p^*(f): \ul\cE\times_{\ul\cB} \ul\cB^{\prime\prime} \to \ul\cB^{\prime\prime} \to \ul\cB$
naturally maps to $f$ as a map of $X$-fibrations over $\ul\cB$.

Consequently, $p^*$ of (\ref{eqn:p*}) is bijective.
\end{lemma}

\vskip .1in

 \begin{prop}
 \label{prop:functor}
 Consider a pointed, connected simplicial set $X$.
 Sending a pointed-connected special $\cF$-space $\ul\cB$ over $\ul\cN$ to $X{\text-}fib(\ul\cB)$
 and using the bijections of Lemmas \ref{lem:iota!}.and \ref{lem:p*}, 
 we obtain the functor
 $$X{\text-}fib(-): Ho((\cF{\text -}spec/\ul\cN)_*)^{op} \quad \to \quad (sets). $$
 
 Consequently, sending $\phi: \ul\cB \dashrightarrow \ul\cB G(X)$ to $\phi^*(\pi_X) \in X{\text-}fib(\ul\cB)$
 determines a well defined natural transformation of functors
 \begin{equation}
 \label{eqn:upper-X}
 (-)^*(\pi_X): Hom_{Ho((\cF{\text -}spec/\ul\cN)_*)}(-,\ul\cB G(X)) \quad \to \quad X{\text-}fib(-).
 \end{equation}
 \end{prop}

\begin{proof}
If $\phi: \ul\cB \dashrightarrow \ul\cB_1$ is a map in $Ho((\cF{\text -}spec/\ul\cN)_*)$, 
let $p: \ul\cB^{\prime\prime} \twoheadrightarrow \ul\cB$ be a cofibrant replacement and let 
$\iota:\ul\cB_1 \hookrightarrow \ul\cB_1^\prime$ be a fibrant replacement.  
Represent $\phi$ by a chain of the form 
$\ul\cB \ \stackrel{p}{\leftarrow} \ \ul\cB^{\prime\prime} \ \stackrel{f}{\to} \ \ul\cB_1^\prime \ \stackrel{\iota}{\leftarrow} \ \ul\cB_1.$
We define
$$X{\text-}fib(\phi) \ \equiv \ (p \circ) \circ  f^* \circ  \iota_!: \ X{\text-}fib(\ul\cB_1) \quad \to \quad  X{\text-}fib(\ul\cB),$$
which can be viewed as $(p^*)^{-1} \circ f^* \circ (\iota^*)^{-1}$ by Lemmas \ref{lem:iota!} and \ref{lem:p*}.  

Granted that
two choices of $f: \ul\cB^{\prime\prime} \to \ul\cB^\prime_1$ representing $\phi$ are related by a homotopy from
a cyclinder object for $\ul\cB^{\prime\prime} $ to $\ul\cB^\prime_1$, the pull-back $f^*$ is independent of this 
choice.  To verify that the composition is independent of choices of cofibrant and fibrant replacements is a 
straight-forward argument involving the comparison of two such choice by ``dominating" them with a third.
\end{proof}

 \vskip .1in

The following definition of principal $d$-simplices is similar that of \cite[Defn 4.1]{F80}.  Because
$d$-simplices of $G(X)$ are maps $X \times \Delta[d] \to Sin(|X|)$, our definition of a principal
$d$-simplex of $f_I$ is a map $\phi: X^I \times \Delta[d]
 \to Sin(|\ul\cE_I |)$ projecting to a $d$-simplex of $Sin(|\ul\cB_I |)$
 rather than simply a map $\phi: X^I \times \Delta[d] \to \ul\cE_I$ projecting to a $d$-simplex of $\ul\cB_I$.

\begin{defn}
\label{defn:principal-simplex}
Let $X$ be a pointed, connected simplicial set.
Consider  an $X$-fibration $f: \ul\cE \ \to \ \ul\cB$ (with section $s:\ul\cB \to \ul\cE$ implicit)
for some pointed-connected special $\cF$-space $\ul\cB$ over $\ul\cN$.
 For $I = (i_,\ldots,i_n) \in \ul\cN({\bf n})$, we say that a map $\phi: X^I \times \Delta[d]
 \to Sin(|\ul\cE_I |)$ is a principal $d$-simplex of $f_I$ provided that
 \begin{enumerate}
 \item
 The composition $f_I \circ \phi: X^I \times \Delta[d] \to Sin(|\ul\cE_I |) \to Sin(|\ul\cB_I |)$ factors though the
 projection, $X^I \times \Delta[d] \to \Delta[d] \to Sin(|\ul\cB_I |)$ whose associated map 
 $X^I \times \Delta[d] \ \to \ Sin(|\ul\cE_I |) \times_{Sin(|\ul\cB_I |)} \Delta[d]$ restricted to each fiber
 over $\Delta[d]$  is weakly homotopic to the canonical pointed map $X^I \to Sin(|X^I|)$.
 \item
 Consider a non-decreasing map $\alpha: {\bf n} \to {\bf m}$ in $\cF$ and
 set $J= (j_1,\ldots,j_m) \in \ul\cN({\bf m}), \ j_s = \sum_{\alpha(i) = s} i_s$.
 Then $\phi$ induces $\phi^\alpha: X^J \times \Delta[d] \to Sin(|\ul\cE_{J} |)$ fitting in the commutative square
 \begin{equation}
\xymatrix{
X^I \times \Delta[d] \ar[d]_{\alpha} \ar[r]^-\phi &  Sin(|\ul\cE_I|) \ar[d]^{\alpha} \\
X^J \times \Delta[d] \ar[r]^-{\phi^\alpha} & Sin(|\ul\cE_J|)
}
\end{equation}
such that $\phi^\alpha$ restricted to each fiber over $\Delta[d] \to Sin(|\ul\cB_I |)$ is weakly
homotopy equivalent to the canonical map $X^J\to Sin(|X^J|)$.
 
For $I = (i_,\ldots,i_n) \in \ul\cN({\bf n})$, we define 
\begin{equation}
\label{eqn:prin(f)}
P(f_I) \quad \subset \quad \ul{Hom}(X^I,Sin(|\cE_I |)) 
\end{equation}
to be the subcomplex of  principal simplices of $f_I$ .  
By construction, there is a natural map of simplicial sets $P(f)_I \to Sin(|\ul\cB_I|)$.
 \end{enumerate}
\end{defn}

\vskip .1in

We state the following analog of Example \ref{ex:G(X)-act} for the right action of $G(X^I)$ on $P(f_I)$.

\begin{lemma}
\label{lem:prin-X}
Consider a pointed, connected simplicial set $X$ and a pointed-connected special $\cF$-space $\ul\cB$ over $\ul\cN$.
Let $f: \ul\cE \to \ul\cB$ be an $X$-fibration (with section $s:\ul\cB \to \ul\cE$ implicit).  Consider the  right action
of $G(X^I)$ on $P(f_I)$ for some $I \in \ul\cN({\bf n})$
\begin{equation}
\label{eqn:prin-act}
P(f_I) \times G(X^I) \quad \to \quad P(f_I)
\end{equation}
which sends the pair
$(\phi: X^I \times \Delta[d] \to Sin(|\ul\cE_I |),\ \theta: X^I \times \Delta[d] \to Sin(| X^I |))$ to 
the composition
$$X^I \times \Delta[d] \stackrel{\theta,pr_2}{\to} Sin(| X^I |) \times \Delta[d] \stackrel{Sin\circ |\phi|}{\to} 
Sin(|(Sin(|\cE_I |)|))   \to Sin(|\cE_I |).$$

The action of (\ref{eqn:prin-act}) fits in commutative squares involving compositions 
and smash products for any non-decreasing  $\alpha: {\bf n} \to {\bf m} \in \cF$,
$I \in \ul\cN({\bf n})$, and $J = \alpha(I)$
\begin{equation}
\label{eqn:prin-compat}
\xymatrix{
P(f_{I}) \times G(X^I) \ar[r] \ar[d]_\alpha & P(f_I) \ar[d]^\alpha \\
P(f_{J}) \times G(X^J) \ar[r]  & P(f_J). 
}
\end{equation}

\end{lemma}

\vskip .1in

We relate various $X$-fibrations which we have seen in examples, following \cite[Thm 9.2]{May75}.

\begin{prop}
\label{prop:2-sided}
Consider a pointed, connected simplicial set $X$, a pointed-connected special $\cF$-space $\ul\cB$ over $\ul\cN$,
and an $X$-fibration $f: \ul\cE \to \ul\cB$ (with section $s:\ul\cB \to \ul\cE$ implicit).
The diagrams (\ref{eqn:prin-compat}) for all $\ul n$ and $I \in \ul\cN(\ul n)$ 
determine a pointed-connected special $\cF$-space $\ul\cB (P(f),G(X))$ and (using the 2-sided bar construction)
the $X$-fibration $\pi_{P(f),X}: \ul\cB (P(f),G(X),X) \to \ul\cB (P(f),G(X))$.  

Moreover, $\pi_{P(f),X}$ fits in the following commutative diagram of $\cF$-spaces
\begin{equation}
\label{eqn:2-sided-func}
\xymatrix{
\ul\cE \ar[d]_f \ar[r] &Sin(|\ul\cE|) \ar[d]^{Sin(|f|)} & \ar[l]_-{\tilde p_f} \ul\cB (P(f),G(X),X)
 \ar[d]^{\pi_{P(f),X}} \ar[r]^-{\tilde q_f} &   \ul\cB (G(X),X) \ar[d]_{\pi_X} &  \\
\ul\cB \ar[r]_\iota & Sin(|\ul\cB|)  & \ar[l]^-{p_f} \ul\cB (P(f),G(X)) \ar[r]_{q_f} &  \ul\cB G(X) 
}
\end{equation}
whose squares are maps of $X$-fibrations and whose left and middle horizontal maps are weak equivalences.
\end{prop}

\begin{proof}
For any $I = (i_1,\ldots,i_n) \in \ul\cN({\bf n})$, Example \ref{ex:G(X)-act} and Lemma \ref{lem:prin-X} determine
commutative diagram
 of pointed simplicial sets
\begin{equation}
\label{eqn:2-sided}
\xymatrix{
Sin(| \ul\cE |)_I \ar[d]^{f_I}  & \ar[l]^-{\tilde p_{f,I}} \ul\cB(P(f),G(X),X) _I
 \ar[d]^{\pi_{P(f),X,I}} \ar[r]^-{\tilde q_{f,I}}  &   
 B(G(X),X)_I \ar[d]_{\pi_{X,I}} &  \\
Sin(|\ul\cB|)_I  & \ar[l]^-{p_{f,I}} B(P(f),G(X))_I \ar[r]_{q_{f,I}} &  BG(X)_I.
}
\end{equation}
The simplicial set $\ul\cB(P(f),G(X),X) _I $ given by the 2-sided bar construction 
has set of degree $t$ simplices equal to
$$P(f_I)_t \times G(X^I)^{\times t} \times Sin(|X^I|)_t \times (\prod_{S\subset {\bf n}} \Sigma_{|S|})^{t+1}.$$

 The right horizontal maps of (\ref{eqn:2-sided}) are projections (sending $P(f_I)$ to the base point),  
so that the right squares of (\ref{eqn:2-sided}) are cartesian.  The lower left map, $p_{f,I}$, 
is given by projection to $P(f)_I$ followed by $P(f)_I  \to Sin(|\ul\cB_I|)$.
The upper left map $\tilde p_{f,I}$ utilizes the right action of $G(X^I)$ on
$P(f_I)$ and the left action of $G(X^I)$ on $Sin(|X^I|)$.  
Then $\tilde p_{f,I}$ sends a $d$-simplex of this simplicial set represented (with $\Delta[d]$ omitted)
as \ $(\phi:X^I \to Sin(|\cE_I|),\ (g_1,\ldots,g_t) \in G(X^I)^{\times t}, 
\ y \in Sin|(X^I|), \ \ul \sigma \in (\prod_{S\subset {\bf n}} \sigma_{|S|})^{t+1}$  
to the $d$-simplex of $Sin(| \ul\cE |)_I$ represented as 
$(\phi\circ g_1)(g_t \circ y)$, where the actions $(-)\circ g_1$ and $g_t \circ (-y)$ are given
by Example \ref{ex:G(X)-act} and Lemma \ref{lem:prin-X}, ``twisted"  by $\ul \sigma$.
So defined, the left squares of (\ref{eqn:2-sided}) are commutative and homotopy cartesian. 

The assertion that the diagrams (\ref{eqn:2-sided}) determine the commutative diagram of $\cF$-spaces
(\ref{eqn:2-sided-func}) is the assertion that maps of (\ref{eqn:2-sided}) are functorial
with respect to maps $\alpha: {\bf n} \to {\bf m} \in \cF$ (sending $I\in \ul\cN({\bf n})$ to 
$J= \alpha(I) \in \ul\cN({\bf m})$).  This is essentially implicit in the construction of these $\cF$-spaces
which entails the compatibility of $G(X^I)$-actions with smash products.  
Since the squares of (\ref{eqn:2-sided}) are cartesian, the squares of  (\ref{eqn:2-sided-func})
are maps of $X$-fibrations.

For any simplex $\Delta[d] \to \ul\cB_I$, there is a $G(X^I)$-equivariant weak equivalence \\
$G(X^I) \times \Delta[d]  \to P(f_I)\times_{\ul\cB_I} \Delta[d]$. 
This implies that the lower left map of (\ref{eqn:2-sided-func}) has contractible fibers 
and thus is a a weak equivalence.  The 
left square of (\ref{eqn:2-sided}) is an equivalence of $X$-fibrations in view of
 the fact that for any simplical set $Y$ the 
natural map $Y \to Sin(|Y|)$ is a weak equivalence.
\end{proof}

\vskip .1in

Propositions \ref{prop:functor} and \ref{prop:2-sided} lead us to the following representability
theorem for $X$-fibrations.

\begin{thm} (cf. \cite[Thm 6.1]{F80})
\label{thm:X-universal}
Let $X$ be a pointed, connected simplicial set.
 The  natural transformation of Proposition \ref{prop:functor}
 \begin{equation}
(-)^*(\pi_X):  Hom_{Ho((\cF{\text -}spec/\ul\cN)_*)}(-, \ul\cB G(X)) \quad \to \quad X{\text-}fib(-) 
\end{equation}
is an isomorphism of functors with inverse
\begin{equation}
\label{eqn:nat-trans}
\rho_X:  X{\text-}fib(-) \quad \to \quad Hom_{Ho((\cF{\text -}spec/\ul\cN)_*})(-, \ul\cB G(X))
\end{equation}
sending the equivalence class  in $X{\text-}fib(\ul\cB)$ of an $X$-fibration $f: \ul\cE \to \ul\cB$  to the homotopy class 
of the composition 
\begin{equation}
\label{eqn:rho}
\rho_X(f) \quad \equiv \quad q_f \circ p_f^{-1} \circ \iota: \ul\cB \dashrightarrow \ul\cB G(X)
\end{equation}
 as in (\ref{eqn:2-sided-func}).

If $f: \ul\cE \to \ul\cB$ is an $X$-fibration, then a representative of $\rho_X(f): \ul\cB \dasharrow \ul\cB G(X)$
is called a classifying map for $f$.
\end{thm}

\begin{proof}
The diagram (\ref{eqn:2-sided-func}) displays an arbitrary $X$-fibration $f: \ul\cE \to \ul\cB$ as equivalent to the pull-back along 
$\rho_X(f): \ul\cB \dashrightarrow \ul\cB G(X)$ of $\pi_X$.   Thus, 
$(-)^*(\pi_X):  Hom_{Ho((\cF{\text -}spec/\ul\cN)_*)}(\ul\cB, \ul\cB G(X)) \quad \to \quad X{\text-}fib(\ul\cB)$ \ is 
surjective.

If $f \equiv \phi^*(\pi_x)$ is equivalent to $f^\prime \equiv \phi^{\prime*}(\pi_X)$, 
then this equivalence gives a chain of diagrams of $X$-fibrations covering the identity
of $\ul\cB$ in $Hom_{Ho((\cF{\text -}spec/\ul\cN)_*)}(\ul\cB,\ul\cB)$.  Focusing on the base of these chains 
of $X$-fibrations, we see that \\
$\phi = \phi^\prime \in Hom_{Ho((\cF{\text -}spec/\ul\cN)_*)}(\ul\cB, \ul\cB G(X))$.
In other words, \\
$(-)^*(\pi_X):  Hom_{Ho((\cF{\text -}spec/\ul\cN)_*)}(\ul\cB, \ul\cB G(X)) \quad \to \quad X{\text-}fib(\ul\cB)$
\ is also injective.

Observe that if $f: \ul\cE \to \ul\cB$ is an $X$-fibration whose equivalence class is represented by
$\rho_X(f): \ul\cB \dashrightarrow \ul\cB G(X)$, then $f$ is equivalent to $(\rho_X(f))^*(\pi_X) \ = \
(\rho_X \circ (-)^*(\pi_X))(f)$ (by the definitions of these constructions).
\end{proof}

\vskip .2in


\section{$\bZ/\ell$-complete $X$-fibrations}
\label{sec:ell-complete}

We proceed to provide $\bZ/\ell$-complete analogues of 
the constructions and results of Section \ref{sec:X-fib}.   
In Definition \ref{defn:ell-fibration}, we define a
 $\bZ/\ell$-complete $X$-fibration to be a  map of $\cF$-spaces $f:\ul\cE \ \to \ \ul \cB$ over
$\ul\cN$ equipped with a section $s: \ul\cB \to \ul\cE$ such that the fibers of 
$f_I: \ul\cE_I \ \to \ul \cB_I$ are (pointed) homotopy equivalent
to $(X^I)_\ell $ for $I = (i_1,\ldots,i_n) \in \ul\cN(\bf n)$.
The ``classifying $\cF$-space" for such fibrations involves monoids $G_\ell(X^I)$
of mod-$\ell$ equivalences $X^I \to (X^I)_\ell$.

As in Example \ref{ex:G(X)-act} where we considered maps from $X$ to $Sin(| X |)$, 
it is important that we consider maps whose source is often a product of smash products (and thus not a 
smash product of $\bZ/\ell$-completions) and whose target is $\bZ/\ell$-complete.

We begin by observing that if $\ul\cB$ is a pointed-connected special $\cF$-space over $\ul\cN$, 
then $\ul\cB_\ell$ defined by setting $(\ul\cB_\ell)_I$ equal to $(\ul\cB_I)_\ell$ 
is also a pointed connected $\cF$-space over $\ul\cN$.

\vskip .1in

\begin{defn}
\label{defn:smash-ell-good}
We recall that a pointed, connected simplicial set $X$ is called $\ell$-good if the 
canonical $\bZ/\ell$-completion map $X \to X_\ell$ is a mod-$\ell$-equivalence
(i.e., induces an isomorphism in mod-$\ell$ homology).  
We say that a pointed simplicial set $X$ is ``smash $\ell$-good" if the $m$-fold smash 
product $X^m$ of $X$ is $\ell$-good for all $m > 0$.

Assume that $X$ is smash $\ell$-good.  For $I = (i_1,\ldots i_n) \in \ul\cN({\bf n})$, we  denote by
\begin{equation}
\label{eqn:action1}
G_\ell(X^I) \ \quad \hookrightarrow \quad \ul{Hom}_*(X^I,(X^I)_\ell)
\end{equation}
the (simplicial) function complex consisting of those $t$-simplices $\sigma: X^I \times \Delta[t] \to (X^I)_\ell$
which are mod-$\ell$ equivalences.   We denote by $G_\ell^o(X^I) \subset G_\ell(X^I)$ the subcomplex
of those $t$-simplices $\sigma: X^I \times \Delta[t] \ \to (X^I)_\ell$ with the property that
$\sigma_\ell: (X^I \times \Delta[t])_\ell \ \to ((X^I)_\ell)_\ell \to (X^I)_\ell$ is homotopic to the identity.
\end{defn}

\vskip .1in

We state the $\bZ/\ell$-completion analogue of Definition \ref{defn:G(X)} and Example \ref{ex:G(X)-act}.

\begin{prop}
\label{prop:G-ell(X)-act}
Let $X$ be a pointed simplicial set which is smash $\ell$-good.
Consider some $I = (i_1,\ldots,i_n) \in \ul\cN({\bf n})$.
\begin{enumerate}
\item 
$G_\ell(X^I)$ is a simplicial monoid and $G_\ell^o(X^I) \ \subset \ G_\ell(X^I)$ 
is a simplicial sub-monoid, where the product structure is induced by composition of functions.
\item 
There is a natural action $G_\ell(X^I) \times X^I \ \to \  (X^I)_\ell$ 
determining an inclusion of simplicial monoids
$$G_\ell(X^I) \quad \hookrightarrow \quad \ul{Hom}_*(X^I,(X^I)_\ell).$$
\item 
If $\alpha: {\bf n} \to {\bf m}$ non-decreasing,  $I \in \cF({n})$ and $J = (j_1,\ldots j_m)$ with $j_s = \sum_{\alpha(i) = s}i$,
the action of (2) fits in a commutative square involving compositions and smash products
\begin{equation}
\label{eqn:smash}
\xymatrix{
G_\ell(X^I) \times X^I \ar[r] \ar[d]_\alpha & (X^I)_\ell \ar[d]^\alpha \\
G_\ell(X^J) \times X^J \ar[r] & (X^J)_\ell .
}
\end{equation}
\end{enumerate}
\end{prop}

\begin{proof}
The ``composition product" $G_\ell(X^I) \times G_\ell(X^I) \ \to \ G_\ell(X^I)$
is defined by sending $\gamma_1, \gamma_2: \Delta[d] \times X^I \to (X^I)_\ell$ to the 
composition
\begin{equation}
\Delta[d] \times X^I \ \stackrel{pr_1\times \gamma_2}{\to} \ \Delta[d]\times (X^I)_\ell \to 
(\Delta[d]\times X^I)_\ell \ \stackrel{(\gamma_1)_\ell}{\to} \ ((X^I)_\ell)_\ell \ \stackrel{\eta}{\to} \ (X^I)_\ell \ ,
\end{equation}
where $\eta: (\bZ/\ell)_\infty \circ (\bZ/\ell)_\infty \to (\bZ/\ell)_\infty$ is part of the natural 
triple structure (see, for example, \cite{MacL}). 

The map $\ul{Hom}_*(X^I,X^I_\ell) \ \to \ \ul{Hom}_*((X^I)_\ell,(X^I)_\ell)$
commutes with composition thanks to the ``idempotence" of $\eta$.  This readily 
implies the commutativity of (\ref{eqn:smash}).
\end{proof}

\vskip .1in

The following definition is the $\bZ/\ell$-complete analog of Definition \ref{defn:X-fibration}.
We shall often suppress mention of the section 
$s: \ul\cB \to \ul\cE$ of an $X$ fibration ($f: \ul\cE \ \to \ \ul\cB, s: \ul\cB \to \ul\cE)$.

\begin{defn}
\label{defn:ell-fibration}
Let $X$ be a pointed, connected simplicial set which is smash $\ell$-good.
A $\bZ/\ell$-complete $X$-fibration over a pointed-connected special $\cF$-space 
\\ ${\ul \cB}: \cF \to (s.sets_*)$
is a  pair of maps \ ($f: \ul\cE \ \to \ \ul\cB, s: \ul\cB \to \ul\cE)$ \ of special $\cF$-spaces 
over $\ul \cN$ with $s \circ f = id$  satisfying the following conditions:
\begin{enumerate}
\item
For each $n > 0$ and $I =(i_1,\ldots,i_n) \in \ul\cN({\bf n})$, $f_I: \ul\cE_I\to \ul\cB_I $ is a fibration with fibers which
are pointed homotopy equivalent to $(X^I)_\ell$.
\item
For any non-decreasing $\alpha: {\bf n} \to {\bf  m} \in \cF$,  $I \in \ul\cN({\bf n})$,  $J = \alpha(I) \in \ul\cN({\bf m})$,
and $\Delta[d] \to \ul\cB_I$, the restriction of $\alpha$ above $\Delta[d]$, 
$$\alpha_{I,\Delta[d]}: \ul\cE_I \times_{\ul\cB_I} \Delta[d] \ \to \ \ul\cE_J \times_{\ul\cB_J} \Delta[d],$$
is homotopic to $(\wedge_{\alpha,I})\times id: (X^I)_\ell \times \Delta[d] \quad \to \quad (X^J)_\ell \times \Delta[d].$
\end{enumerate}

A map of $\bZ/\ell$-complete $X$-fibrations
$(\tilde \phi,\phi)$ from $(f:\ul\cE \to \ul\cB, \ s:\ul\cB \to \ul\cE)$  to
$(f^\prime:\ul\cE^\prime \to \ul\cB^\prime, s^\prime: \ul\cB^\prime \to \ul\cE^\prime)$ 
is a commutative square of $\cF$-spaces which restricts for each $I \in \cF({\bf n})$
to a fiber homotopy equivalence of the form (\ref{eqn:F-map}).
\end{defn}

Once again, we shall often suppress mention of the section 
$s: \ul\cB \to \ul\cE$ of a $\bZ/\ell$-completed $X$ fibration ($f: \ul\cE \ \to \ \ul\cB, s: \ul\cB \to \ul\cE)$.

\vskip .1in

The following example is the $\bZ/\ell$-complete analog of $\ul\cB G(X)$.

\begin{ex}
\label{ex:BGX-X-ell}
Assume that $X$ is a smash $\ell$-good, pointed, connected simplicial set.
In a manner parallel to the construction of the $\cF$-space $\ul\cB G(X)$, 
the actions \ $G_\ell(X^I) \times X^I \ \to \ (X^I)_\ell$
for all $I \in \cN({\bf n})$ determine the pointed special $\cF$-space over $\ul\cN$
$$\ul \cB G_\ell(X): \cF \quad \to \quad (s.sets_*)$$
with
$$ \ul\cB G_\ell(X)({\bf 1}) \ = \ \coprod_{n\geq 0} B G_\ell(X^n) \ = \ \ul\cB G_\ell(X)({\bf 1}), \quad
{\ul \cB}G_\ell(X)_I \ = \ 
BG_\ell(X^I) \times \prod_{S\subset {\bf n}} E(\Sigma_{|S|},\Sigma_{|S|}).$$

Similarly, in parallel to the $X$-fibration $\pi_X: \ul\cB(G(X),X) \to \ul\cB G(X)$ of 
Example \ref{ex:G(X)-act}, these actions 
determine the $\bZ/\ell$-complete $X$-fibration
\begin{equation}
\label{eqn:pi-univ-ell}
\pi_{X,\ell}: \ul \cB (G_\ell(X),X_\ell) \quad \to \quad \ul \cB G_\ell(X)
\end{equation}
\begin{equation}
\label{eqn:explicit-X-ell}
{\ul \cB}G_\ell(X)({\bf n})_I \quad \simeq \quad
BG_\ell(X^I) \times \prod_{S\subset {\bf n}} E(\Sigma_{|S|},\Sigma_{|S|}).
\end{equation}

If $\Sigma_n \subset G(X^n)$ for all $n > 0$, then we similarly define
\begin{equation}
\label{eqn:pi-univ-ell-o}
\pi_{X,\ell}^o: \ul \cB (G_\ell^o(X),X_\ell) \quad \to \quad \ul \cB G_\ell^o(X).
\end{equation}
\end{ex}

\vskip .1in

\begin{remark}
\label{rem:contrast}
We emphasize that the construction of $\pi_{X_\ell}: \ul \cB (G_\ell(X),X_\ell) \ \to \ \ul \cB G_\ell(X)$
differs from that of Example \ref{ex:G(X)-act} with $X$ replaced by $X_\ell$.
One difference is that  $(X^m)_\ell$ is not the $m$-fold smash 
product of $X_\ell$; another is that $G_\ell(X) \ \not= \ G(X_\ell)$.
\end{remark}

\vskip .1in

\begin{prop}
\label{prop:X-ell-cases}
Let $X$ be a smash $\ell$-good, pointed, connected simplicial set and consider a
pointed-connected special $\cF$-space $\ul \cB$ over $\ul\cN$.
\begin{enumerate}
\item
The fiber-wise $\bZ/\ell$-completion $(\bZ/\ell)_\infty^\bu(f)$ of an $X$-fibration 
$f: \ul\cE \to \ul\cB$ is a $\bZ/\ell$-complete $X$-fibration over $\ul\cB$.
\item
If $f: \ul\cE \to \ul\cB$ is a $\bZ/\ell$-complete $X$-fibration over $\ul\cB$, then the
fiber-wise $\bZ/\ell$-completion $(\bZ/\ell)_\infty^\bu(f)$ is also a $\bZ/\ell$-complete $X$-fibration over $\ul\cB$
and the canonical map $f \to (\bZ/\ell)_\infty^\bu(f)$ is a fiber homotopy equivalence
of $\bZ/\ell$-complete $X$-fibrations over $\ul\cB$.
\item
If $f: \ul\cE \to \ul\cB$ is an $X$-fibration  and if 
$\pi_1(\ul\cB_I)$ acts nilpotently on $H^*(X^I,\bZ/\ell)$ for all $n > 0$, all $I \in \ul\cN({\bf n})$,
then $(\bZ/\ell)_\infty(f): \ul\cE_\ell \ \to \ \ul\cB_\ell$
is a  $\bZ/\ell$-complete $X$-fibration over $\ul\cB_\ell$.
\end{enumerate}
\end{prop}

\begin{proof}
The first assertion follows from the functoriality of the fiber-wise $\bZ/\ell$-completion
functor sending conditions of Definition \ref{defn:X-fibration} to conditions 
of Definition \ref{defn:ell-fibration}.  To prove assertion (2), observe that
if the fibers of $f_I$  are homotopy equivalent to $(X^I)_\ell$, then the fibers of
$(\bZ/\ell)^\bu(f_I)$ are homotopy equivalent to $((X^I)_\ell)_\ell$ which is naturally 
homotopy equivalent to $(X^I)_\ell$.  Namely, using the natural commutative square 
\begin{equation}
\xymatrix{
X^I \ar[d] \ar[r] &  (X^I)_\ell \ar[d] \\
(X_\ell)^I \ar[r] & ((X_\ell)^I)_\ell
}
\end{equation}
and the fact that the $(\bZ/\ell)$-completion of
$(X^I)_\ell$ of the $\ell$-good simplicial set $X^I$ is also $\ell$-good, we conclude that
if $X$ is smash $\ell$-good then so is $X_\ell$.

Granted the hypotheses of assertion (3), the  Bousfield--Kan``mod-$\bZ/\ell$ fibre lemma" 
(see \cite[II.5.1]{Bo-Kan}) implies  that the natural map from the $(\bZ/\ell)_\infty$-completion of 
the fiber of $f: \ul\cE \to \ul\cB_I$ to the fiber of  $(\bZ/\ell)_\infty(f)$ is a homotopy
equivalence; in other words, the fiber of $(\bZ/\ell)_\infty(f)$ is homotopy equivalent 
to $(X^I)_\ell$.  
\end{proof}

\vskip .1in
\begin{remark}
The $\bZ/\ell$-completed $S^2$-fibrations occurring in (\ref{eqn:base-change-ss}) provide
further examples of $\bZ/\ell$-completed $X$-fibrations which are not included among the 
examples of Proposition \ref{prop:X-ell-cases}.
\end{remark}

\vskip .1in

The following definition of the set of $X_\ell{\text -}fib(\ul\cB)$ equivalence classes of 
$\bZ/\ell$-complete $X$-fibrations over $\ul\cB$ is the evident modification
of the definition of the set $X{\text -}fib(\ul\cB)$ given in Definition \ref{defn:Xfib}.
This uses Lemma \ref{lem:iota!-construct}, stated for $X$-fibrations but whose proof
applies verbatim for $\bZ/\ell$-complete $X$-fibrations.

\begin{defn}
\label{defn:X-ell-fib}
Consider a pointed, connected simplicial set $X$  which is smash $\ell$-good and a 
pointed-connected special $\cF$-space 
$\ul\cB$ over $\ul\cN$.
We define an equivalence relation on $\bZ/\ell$-complete $X$-fibrations over $\ul\cB$  by setting  $f \ \sim \ f^\prime$
if and only if there is a chain of weak equivalences of pointed special $\cF$-spaces over $\ul\cN$
\begin{equation}
\label{eqn:self-chain-ell}
\ul\cB \ \stackrel{p}{\twoheadleftarrow} \quad \ul\cB_1 \quad \stackrel{\phi}{\to} \quad \ul\cB_2  
\quad \stackrel{j}{\hookleftarrow}  \ \quad \ul\cB
\end{equation}
and a  map of $\bZ/\ell$-complete $X$-fibrations $p^*(\ul\cE) \to j_!(\ul\cE^\prime)$ covering $\phi: \ul\cB_1 \to \ul\cB_2$
(with sections)
$$p^*(\ul\cE) \quad \to \quad \iota_!(\ul\cE^\prime),$$
where $p$ is a cofibrant replacement of $\ul\cB$, $j$ is a fibrant replacement of $\ul\cB^\prime$,  
and the image of (\ref{eqn:self-chain-ell})
 in $Hom_{Ho((\cF{\text -}spec/\ul\cN)_*)}(\ul\cB,\ul\cB)$ is the identity.

We denote by $X_\ell{\text -}fib(\ul\cB)$ the set of equivalence classes of $\bZ/\ell$-complete $X$-fibrations over $\ul\cB$.
\end{defn}

\vskip .1in

The  proofs  of Lemmas \ref{lem:iota!} and \ref{lem:p*}
apply verbatim to the prove their $\bZ\ell$-complete analogs.  Thus, the proof of Proposition \ref{prop:functor}
can be repeated to prove the following proposition.

 \begin{prop}
 \label{prop:functor-ell}
 Consider a pointed, connected simplicial set $X$ which is smash $\ell$-good.
 Sending a pointed-connected special $\cF$-space $\ul\cB$ over $\ul\cN$ to $X_\ell{\text-}fib(\ul\cB)$
 and using bijections of $\bZ/\ell$-complete analogs of Lemmas \ref{lem:iota!}.and \ref{lem:p*}, 
 we obtain the functor
 $$X_\ell{\text-}fib(-): Ho((\cF{\text -}spec/\ul\cN)_*)^{op} \quad \to \quad (sets). $$
 
 Consequently, sending $\phi: \ul\cB \dashrightarrow \ul\cB G_\ell(X)$ to $\phi^*(\pi_{X,\ell}) \in X{\text-}fib(\ul\cB)$
 determines a well defined natural transformation of functors
 \begin{equation}
 \label{eqn:upper-ell-X}
 (-)^*(\pi_{X,\ell}): Hom_{Ho((\cF{\text -}spec/\ul\cN)_*)}(-,\ul\cB G_\ell(X)) \quad \to \quad X_\ell{\text-}fib(-).
 \end{equation}
 \end{prop}

\vskip .1in

We give the $\bZ/\ell$-complete analogue of the construction of principal $d$-simplices given in
Definition \ref{defn:principal-simplex}.   For a $\bZ/\ell$-complete $X$-fibration $f$
(with section implicit), 
we  define a principal $\bZ/\ell$-complete $d$-simplex
as a map $X^I \times \Delta[d]  \to (\bZ/\ell)_\infty^\bu(\ul\cE_I)$ (rather than a map 
$X^I \times \Delta[d]  \to \ul\cE_I$).

\begin{defn}
\label{defn:principal-ell-simplex}
Assume that $X$ is a pointed, connected simplicial set which is smash $\ell$-good.
Let $f: \ul\cE \ \to \ \ul\cB$ be an $\bZ/\ell$-complete $X$-fibration with $\ul\cB$ a
pointed-connected special $\cF$-space over $\ul\cN$, 
and denote by $(\bZ/\ell)_\infty^\bu(f): (\bZ/\ell)_\infty^\bu(\ul\cE) \to \ul\cB$ the fiberwise 
$\bZ/\ell$-completion of $f$.  Consider some $I = (i_,\ldots,i_n) \in \ul\cN({\bf n})$.   
We say that a map $\phi: X^I \times \Delta[d]
 \to (\bZ/\ell)_\infty^\bu(\ul\cE_I)$ is a $\bZ/\ell$-complete principal $d$-simplex of $f$ provided that
 $\phi$ satisfies the following conditions.
 \begin{enumerate}
 \item
 The composition $f_I \circ \phi: X^I\times \Delta[d] \to (\bZ/\ell)_\infty^\bu(\ul\cE_I) \to \ul\cB_I$ factors though the
 projection, $X^I\times \Delta[d] \to \Delta[d] \to \ul\cB_I$  and determines a $\bZ/\ell$-equivalence \\
 $X^I \times \Delta[d] \ \to \ (\bZ/\ell)_\infty^\bu(\ul\cE_I) \times_{\ul\cB_I} \Delta[d]$.
 \item
  For every non-decreasing map $\alpha: {\bf n} \to {\bf r}$ in $\cF$,  set 
  $J \equiv \alpha(I) = (j_1,\ldots,j_r) \in \ul\cN(r), \ j_s = \sum_{\alpha(i) = s} i_s$.
 Then $\phi$ induces $\phi_\alpha: X^{\alpha(I)} \times \Delta[d] \to (\bZ/\ell)_\infty^\bu(\ul\cE_{\alpha(I)})$ 
 fitting in the commutative square
 \begin{equation}
\xymatrix{
X^I \times \Delta[d] \ar[d]_{\alpha} \ar[r]^-\phi &  (\bZ/\ell)_\infty^\bu(\ul\cE_I) \ar[d]^{\alpha} \\
X^J \times \Delta[d] \ar[r]^-{\phi^\alpha} & (\bZ/\ell)_\infty^\bu(\ul\cE_J)
}
\end{equation}
such that $\phi^\alpha$ induces a mod-$\ell$ equivalence 
$$X^J \times \Delta[d] \ \to \ (\bZ/\ell)_\infty^\bu(\ul\cE_J)  
\times_{\ul\cB_J} \Delta[d].$$
\end{enumerate}
\end{defn}

\vskip .1in

As in Lemma \ref{lem:prin-X} for $X$-fibrations,  Definition \ref{defn:principal-ell-simplex} determines a suitable right 
action of action of $G_\ell(X^I)$ on $P_\ell(f_I)$.  This action uses the natural retraction 
$(\bZ/\ell)_\infty^\bu(-) \circ (\bZ/\ell)_\infty^\bu(-) \ \to \ (\bZ/\ell)_\infty^\bu(-)$ (a restriction
of the natural retraction $(\bZ/\ell)_\infty(-) \circ (\bZ/\ell)_\infty\ \to \ (\bZ/\ell)_\infty$), arising
from a triple structure on $(s.sets_*/Y)$ for any pointed simplicial set $Y$. 

The justification for the following construction is strictly analogous the justification of (\ref{eqn:2-sided}) 
in Proposition \ref{prop:2-sided}, simply replacing the action of $G(X^I)$ on $P(f_I)$. 
by the action of $G_\ell(X^I)$ on $P_\ell(f_I)$. 

\vskip .1in

\begin{construct} 
\label{construct:principal-ell-spaces}
Assume that $X$ is a pointed, connected simplicial set which is smash $\ell$-good and consider 
some $\bZ/\ell$-complete $X$-fibration $f: \ul\cE \ \to \ \ul\cB$ with $\ul\cB$ a 
pointed-connected special $\cF$-space over $\ul\cN$.
There is a commutative diagram of simplicial sets
\begin{equation}
\label{eqn:2-sided-ell}
\xymatrix{
 (\bZ/\ell)_\infty^\bu(\ul\cE_I) \ar[d]_{f_I}  & \ar[l]_-{\tilde p_{f,I}}  B(P_\ell(f_I),G_\ell(X^I),X_\ell^I )
 \ar[d]^{\pi_{P_\ell,X,\ell,I}} \ar[r]^-{\tilde q_{f,I}}  &   
 B(G_\ell(X_\ell^I),X^I ) \ar[d]^{\pi_{X,\ell,I}} &  \\
\ul\cB_I   & \ar[l]^-{p_{f,I}}  B(P_\ell(f_I),G_\ell(X^I)) \ar[r]_-{q_{f,I}} &  BG_\ell(X^I)
}
\end{equation}
for each $I \in \ul\cN({\bf n})$, where the right horizontal maps are projections, the lower left map
is given by projection to $P_\ell(f_I)$ followed by the structure map
$P_\ell(f_I) \to \ul\cB_I$.  The upper left map $\tilde p_{f,I}$ utilizes the right action of $G_\ell(X^I)$ 
on $P_\ell(f_I)$ and the left action of $G_\ell(X^I)$ on $(X^I)_\ell$.  The formulation of $\tilde p_{f,I}$
is exactly parallel to the formulation the upper left arrow of (\ref{eqn:2-sided}) given in the proof
of Proposition \ref{prop:2-sided}.

The vertical maps of (\ref{eqn:2-sided-ell}) are fibrations whose homotopy fibers are homotopy equivalent to $(X^I)_\ell$, the
squares of (\ref{eqn:2-sided-ell}) are homotopy cartesian, 
and the left square is a homotopy
equivalence of fibrations (see the proof of Proposition \ref{prop:2-sided}).
\end{construct}

\vskip .1in

We basically repeat the proof of Proposition \ref{prop:2-sided} to verify the following
$\bZ/\ell$-complete analogue.  

\vskip .1in

\begin{prop}
\label{prop:func-P(f)-ell}
Consider a pointed, connected simplicial set $X$ which is smash $\ell$-good.
Let  $f: \ul\cE \ \to \ \ul\cB$ be a $\bZ/\ell$-complete $X$-fibration with $\ul\cB$ a
pointed-connected special $\cF$-space over $\ul\cN$.
The diagrams of (\ref{eqn:2-sided-ell}) for all $I \in \ul\cN({\bf n}), n > 0$ determine
a commutative diagram of $\cF$-spaces
\begin{equation}
\label{eqn:2-sided-func-ell}
\xymatrix{
 \ul\cE  \ar[r]^-\epsilon \ar[d]_f & (\bZ/\ell)_\infty^\bu(\ul\cE)  \ar[d]^{(\bZ/\ell)_\infty^\bu(f)} &
 \ar[l]_-{\tilde p_{f,\ell}} \ul\cB (P_\ell(f),G_\ell(X),X_\ell) \ar[d]_{\pi_{P_\ell(f),X_\ell}} \ar[r]^-{\tilde q_{f,\ell}} &   
 \ul\cB (G_\ell(X),X_\ell) \ar[d]_{\pi_{X,\ell}}   \\
\ul\cB \ar[r]^= & \ul\cB   & \ar[l]^-{p_{f,\ell}} \ul\cB (P_\ell(f),G_\ell(X)) \ar[r]_-{q_{f,\ell}} &  \ul\cB G_\ell(X)
}
\end{equation}
whose squares are maps of $\bZ/\ell$-complete $X$-fibrations and whose left and middle horizontal maps are weak equivalences.

In other words, (\ref{eqn:2-sided-func-ell}) determines a homotopy class of maps of 
$\bZ/\ell$-complete $X$-fibrations  $f \to \pi_{X_\ell}$.
\end{prop}

\begin{proof}
The map $\epsilon: \ul\cE   \to (\bZ/\ell)_\infty^\bu(\ul\cE)$
appearing in the upper left corner of (\ref{eqn:2-sided-func-ell}) is induced by the natural 
transformation $id(-) \to  (\bZ/\ell)_\infty^\bu(-)$ determining a fiber homotopy equivalence 
$f \to (\bZ/\ell)_\infty^\bu(f)$ by Proposition \ref{prop:X-ell-cases}(2). 

The middle horizontal 
maps determine weak homotopy equivalences because each $(P_\ell(f))_I \to \ul\cB G_\ell(X)_I$
is a ``principal $G_\ell(X^I)$"-fibration.  By inspection, we see that the right horizontal arrows
determine a map of $\bZ/\ell$-completed $X$-fibrations.
\end{proof}

\vskip .1in

Modifying slightly the proof of Theorem \ref{thm:X-universal} by replacing Propositions 
\ref{prop:functor} and \ref{prop:2-sided} by Proposition \ref{prop:functor-ell} and
Construction \ref{construct:principal-ell-spaces} 
yields a proof of the following representability theorem for $\bZ/\ell$-complete $X$-fibrations.

\begin{thm} 
\label{thm:X-ell-universal}
Let $X$ be a pointed, connected simplicial set which is smash $\ell$-good.
The natural transformation of functors from 
 $Ho((\cF{\text -}spec/\ul\cN)_*)^{op}$  to the category  of sets
 \begin{equation}
(-)^*(\pi_{X,\ell}):  Hom_{Ho((\cF{\text -}spec/\ul\cN)_*)}(-, \ul\cB G_\ell(X)) \quad \to \quad X_\ell{\text-}fib(-) 
\end{equation}
sending $\phi: \ul\cB \dashrightarrow \ul\cB G_\ell(X)$ to  $\phi^*(\pi_{X,\ell})$
in $X_\ell{\text-}fib(\ul\cB)$ is an isomorphism of functors with inverse the natural transformation
\begin{equation}
\label{eqn:rho-ell-map}
\rho_{X,\ell}: X_\ell{\text-}fib(-) \quad \to \quad Hom_{Ho((\cF{\text -}spec/\ul\cN)_*})(-, \ul\cB G_\ell(X))
\end{equation}
sending the equivalence class in $X_\ell{\text-}fib(\ul\cB)$ of a $\bZ/\ell$-complete $X$-fibration $f: \ul\cE \to \ul\cB$  
 to the homotopy class of the composition 
\begin{equation}
\label{eqn:rho-ell}
\rho_{X,\ell}(f) \quad \equiv \quad q_{f ,\ell}\circ p_{f,\ell}^{-1}: \ul\cB \dashrightarrow \ul\cB G(X)
\end{equation}
 as in (\ref{eqn:2-sided-func-ell}).

If $f: \ul\cE \to \ul\cB$ is a $\bZ/\ell$-complete $X$-fibration with $\ul\cB$ a pointed-connected special $\cF$-space over $\ul\cN$, 
then a representative of $\rho_{X,\ell}(f)$ is called a classifying map for $f$.
\end{thm}

\vskip .1in

We relate the classifying theorems Theorem \ref{thm:X-universal} and \ref{thm:X-ell-universal} in the following proposition.

\begin{prop}
\label{prop:relate}
Consider be a pointed-connected special $\cF$-space   $\ul\cB$ over $\ul\cN$,
let $f: \ul\cE \to \ul\cB$ be an $X$-fibration over $\ul\cB$, and consider the fiber-wise $\bZ/\ell$-completion,
 $\tilde f \ \equiv \ (\bZ/\ell)^\bu_\infty(f): \widetilde{\ul\cE} \to \ul\cB$.  Then the (homotopy type of the)
 classifying map 
$\rho_{X,\ell}(\tilde f):  \ul\cB \ \dashrightarrow \ \ul\cB G_\ell(X)$ for $(\bZ/\ell)^\bu_\infty(f)$ 
equals $\iota_\ell \circ \rho_X(f): \ul\cB \ \dashrightarrow \ \ul\cB G(X) \ \to \ \ul\cB G_\ell(X).$ 

Consequently, the following square of functors 
 \begin{equation}
\xymatrix{
Hom_{Ho((\cF{\text -}spec/\ul\cN)_*)}(-,\ul\cB G(X)) \ar[d]_-{\iota_\ell \circ(-)} \ar[rr]^-{(-)^*(\pi_X))}
& & X{\text -}fib(-) \ar[d]^{(\bZ/\ell)^\bu_\infty(-)} \\
Hom_{Ho((\cF{\text -}spec/\ul\cN)_*)}(-,\ul\cB G_\ell(X))  \ar[rr]^-{(-)^*(\pi_{X,\ell)}}
& & X_\ell{\text -}fib(-) 
}
\end{equation}
is commutative.
\end{prop}

\vskip .1in

\begin{remark}
\label{rem:ell-local}
The discussion of this section culminating in Theorem \ref{thm:X-ell-universal} applies
essentially verbatim to the ``Bousfield-Kan localization at $\ell$".   Namely, we simply 
replace $(\bZ/\ell)_\infty(-)$ by $(\bZ_{(\ell)}) _\infty(-)$ and $(-)_\ell$ by $(-)_{(\ell)}$.  
 \end{remark}
 
 \vskip .2in


\section{Oriented $\bZ/\ell$-complete $X$-fibrations}
\label{sec:or-ell-complete}

We begin this section with a few basic properties of oriented $X$-fibrations and
oriented $\bZ/\ell$-complete $X$-fibrations.   For our mail results, the relevance
of orientations arises in the conclusions of Corollary \ref{cor:cover} relating the
classifying $\cF$-spaces $\cB G^o(X)$ and $\cB G^o_\ell(X)$ for $X$ a finite, nilpotent
simplicial set.

 \vskip .1in

\begin{defn}
\label{defn:orient}
Let $X$ be a pointed, connected simplicial set satisfying the condition that $\Sigma_n \subset G^o(X^n)$
(i.e., that permuting factors
of the smash product is weakly homotopy equivalent to the identity) for each $n > 0$.  Consider a
pointed-connected special $\cF$-space $\ul\cB$ over $\ul\cN$.
An  $X$-fibration of $\cF$-spaces, $(f:\ul\cE \to \ul\cB, \ s: \ul\cB \to \ul\cE)$ is said
to be oriented if its classifying map $\rho_X(f): \ul\cB \ \dashrightarrow \ {\ul\cB}G(X)$
is equipped with a factorization of the form $\iota \circ \rho_X^o(f): \ul\cB \ \dashrightarrow \ {\ul\cB}G^o(X) \to \ {\ul\cB}G(X).$

We define $X{\text -}fib^o(\ul\cB)$ to be equivalence class of oriented $X$-fibrations, where we set the 
oriented $X$-fibrations $f, \ f^\prime$  equivalent if their factorizations $ \rho_X^o(f), \ 
 \rho_X^o(f^\prime)$ are equal in $Hom_{Ho((\cF{\text -}spec/\ul\cN)_*)}(-, \ul\cB G(X))$.
\end{defn}

\vskip .1in

We obtain the following as a corollary of Theorem \ref{thm:X-universal}.

\begin{cor}
\label{cor:X-fib-o}
Let $X$ be a pointed, connected simplicial set satisfying the condition that $\Sigma_n \subset G^o(X^n)$.
Then the following natural transformations are mutually inverse:
\begin{equation}
(-)^*(\pi^o_X):  Hom_{Ho((\cF{\text -}spec/\ul\cN)_*)}(-, \ul\cB G^o(X)) \quad \to \quad X{\text-}fib^o(-),
\end{equation}
\begin{equation}
\label{eqn:nat-trans-o}
\rho^o_X:  X{\text-}fib^o(-) \quad \to \quad Hom_{Ho((\cF{\text -}spec/\ul\cN)_*})(-, \ul\cB G^o(X)).
\end{equation}
\end{cor}

\vskip .1in

 We essentially copy Definition \ref{defn:orient} and Corollary \ref{cor:X-fib-o} to give the $\bZ/\ell$-complete
 analogs.
 
 \begin{defn}
\label{defn:orient-ell}
Let $X$ be a pointed, connected simplicial set satisfying the condition that $\Sigma_n \subset G^o(X^n)$
and assume further that $X$ is smash $\ell$-good.  Consider a
pointed-connected special $\cF$-space $\ul\cB$ over $\ul\cN$.
Then a $\bZ/\ell$-complete $X$-fibration $(f: \ul\cE \to \ul\cB, \ s: \ul\cB \to \ul\cE)$  
 is said to be oriented if its classifying map $\rho_{X,\ell}(f): \ul\cB \ \dashrightarrow \ {\ul\cB}G_\ell(X)$
is equipped with a factorization of the form 
$\iota \circ \rho^o_{X,\ell}(f): \ul\cB \ \dashrightarrow \ {\ul\cB}G_\ell^o(X) \to \ {\ul\cB}G_\ell(X).$

We define $X_\ell{\text -}fib^o(\ul\cB)$ to be equivalence class of oriented $\bZ/\ell$-complete $X$-fibrations, where we set the 
oriented $\bZ/\ell$-complete $X$-fibrations $f, \ f^\prime$  equivalent if their factorizations $ \rho^o_{X,\ell}(f), \ 
 \rho^o_{X,\ell}(f^\prime)$ are equal in $Hom_{Ho((\cF{\text -}spec/\ul\cN)_*)}(-, \ul\cB G_\ell^o(X))$.
\end{defn}

\vskip .1in

\begin{cor}
\label{cor:X-fib-o-ell}
Let $X$ be a pointed, connected simplicial set satisfying the condition that $\Sigma_n \subset G^o(X^n)$
and assume further that $X$ is smash $\ell$-good. 
Then the following natural transformations are mutually inverse:
\begin{equation}
(-)^*(\pi^o_{X,\ell}):  Hom_{Ho((\cF{\text -}spec/\ul\cN)_*)}(-, \ul\cB G_\ell^o(X)) \quad \to \quad X_\ell{\text-}fib^o(-),
\end{equation}
\begin{equation}
\label{eqn:nat-trans-ell}
\rho^o_{X,\ell}:  X_\ell{\text-}fib^o(-) \quad \to \quad Hom_{Ho((\cF{\text -}spec/\ul\cN)_*})(-, \ul\cB G_\ell^o(X)).
\end{equation}
\end{cor}

\vskip .1in

The following is an analog of Proposition \ref{prop:relate}.

\begin{cor}
\label{cor:relate-o}
Let $X$ be a pointed, connected simplicial set which is smash $\ell$-good
and consider a pointed-connected special $\cF$-space
$\ul\cB$ which satisfies the condition that $\pi_1(\ul\cB_I) = 0$ 
for all $I \in \ul\cN(\bf n)$, all $n > 0$. 
Let $f: \ul\cE \ \to \ \ul\cB$ be an oriented $X$-fibration with orientation
$\rho^o_X(f):  \ul\cB  \ \to \ \ul\cB G^o(X)$.
Then $(\bZ/\ell)_\infty(f): \ul\cE_\ell \ \to \ \ul\cB_\ell$
is a $\bZ/\ell$-complete $X$-fibration with orientation
$$\lambda \circ (\rho^o_X(f))_\ell:  \ul\cB_\ell \ \to \ \ul\cB G^o(X)_\ell \ \to \ \ul\cB G^o_\ell(X).$$

Let $(\cF{\text -}\widetilde{spec}/\ul\cN)_*$ denote the full subcategory of pointed-connected 
special $\cF$-spaces $\ul\cB$ with the property that $\pi_1(\ul\cB_I) = 0$ 
for all $I \in \ul\cN(\bf n)$, all $n > 0$.  Then the following square of contravariant functors
from $Ho((\cF{\text -}\widetilde{spec}/\ul\cN)_*)$ to (sets)
 \begin{equation}
\xymatrix{
Hom_{Ho((\cF{\text -}\widetilde{spec}/\ul\cN)_*)}(-,\ul\cB G^o(X)) \ar[d]_-{\lambda \circ(-)_\ell} \ar[rr]^-{(-)^*(\pi^o_X))}
& & X{\text -}fib^o(-) \ar[d]^{(\bZ/\ell)_\infty(-)} \\
Hom_{Ho((\cF{\text -}\widetilde{spec}/\ul\cN)_*)}((-)_\ell,\ul\cB G^o_\ell(X))  \ar[rr]^-{(-)^*(\pi^o_{X,\ell}))}
& & X_\ell{\text -}fib^o(-) 
}
\end{equation}
commutes.
\end{cor}

\begin{proof}
By Proposition \ref{prop:X-ell-cases}(3), $(\bZ/\ell)_\infty(f)$ is a $\bZ/\ell$-complete $X$-fibration.
The fact that the classifying map for this  $\bZ/\ell$-complete $X$-fibration is oriented by
$\lambda \circ (\rho^o_X(f))_\ell$ follows from the commutativity of the following diagram
of $\cF$-spaces
\begin{equation}
\label{eqn:comm-cF-2}
\xymatrix{
 (\ul\cE)_\ell \ar[d]_{f_\ell} \ar[r] & \cB (G^o(X),X)_\ell \ar[r]  \ar[d]^{(\pi_{X,o})_\ell} &  \ul \cB (G^o_\ell(X),X_\ell) 
\ar[d]^{\pi_{X,\ell,o}} \ar[r] & \cB (G_\ell(X),X_\ell) \ar[d]^{\pi_{X,\ell}}\\
\ul\cB_\ell \ar[r]_-{(\phi_f)_\ell} & \ul \cB G^o(X)_\ell \ar[r]_{{\lambda}} & \ul \cB G^o_\ell(X) \ar[r] & \ul \cB G_\ell(X) .
}
\end{equation}
\end{proof}

\vskip .1in

The following proposition applies in particular to $X$ with $|X| \simeq S^2$.
The assertions are basically given in \cite[Prop VI.7.1]{Bo-Kan}. 
 We remind the reader
that a pointed, connected topological space $T$ is said to be nilpotent if 
$\pi_1(T)$ acts nilpotently on $\pi_i(T)$ for all $i \geq 1$.

\begin{prop}
\label{prop:fingen}  \cite[Prop 7.1]{Bo-Kan}
Let $X$ be a finite, nilpotent simplicial set.
\begin{enumerate}
\item
The homotopy groups of $G^o(X)$ are finitely generated abelian groups.
\item
The map $\pi_*(G^o(X)) \ \to \ \pi_*((G^o(X)_\ell)$ is given by tensoring with $\bZ_\ell$.
\item
The map $\pi_*(G^o(X)) \ \to \ \pi_*(G_\ell^o(X))$ is given by tensoring with $\bZ_\ell$.
\item
The natural map $(G^o(X))_\ell \ \to \ G^o_\ell(X)$ is a homotopy equivalence.
\end{enumerate}
\end{prop}

\begin{proof}
Assertion (1) is  a well known consequence of the fundamental theorem of J.-P. Serre \cite{Serre}
concerning the homotopy groups of spheres.  (See \cite[Prop 7.1(i]{Bo-Kan}.)
Assertion (2) is given by \cite[Ex VI.5.2]{Bo-Kan}.

Assertion (3) is given by \cite[Prop VI.7.1.(iii)]{Bo-Kan}
Together, assertions (2) and (3) imply that the natural map $(G^o(X))_\ell \ \to \ G^o_\ell(X)$ is a
weak equivalence.  Since this is a map of Kan complexes, it is a homotopy equivalence.
\end{proof}

\vskip .1in

Proposition \ref{prop:fingen}(4) in conjunction with the explicit descriptions of 
$\ul \cB G^o(X))_I$  for each $I \in \ul\cN({\bf n})$ and of $\cB G^o_\ell(X)_I$ 
implies the following corollary.  We remind the reader that the smash product of 
any two pointed, connected simplicial sets is simply connected.
  
 \begin{cor}
\label{cor:cover}
Let $X$ be a finite nilpotent simplicial set.
If $\Sigma_n \subset G^o(X^n)$ for all $n>0$,
then the natural map \ $\cB G^o(X)  \to \cB G^o_\ell(X)$ \ induces a homotopy equivalence
of $\cF$-spaces over $\ul \cN$
\begin{equation}
\label{eqn:S2-hom-equiv}
\lambda: \ul \cB G^o(X)_\ell \quad \to \quad  \ul \cB G^o_\ell(X).
\end{equation}

Moreover, the natural map \ $\ul\cB G^o(X),X)_\ell \ \to \ \ul\cB(G_\ell^o(X),X_\ell)$ \ covers $\lambda$,
determining a map of oriented $\bZ/\ell$-complete $\cF$-spaces 
$$(\pi_X^o)_\ell: \ul\cB G^o(X),X)_\ell \to \cB G^o(X)  \quad \to \quad \pi_{X,\ell}^o: \ul\cB(G_\ell^o(X),X_\ell) \to \ul \cB G^o_\ell(X).$$
\end{cor}

\begin{proof}
The first assertion follows from Proposition \ref{prop:fingen}(4).  The second assertion
follows from Proposition \ref{prop:X-ell-cases}(3) and the fact that each $BG^o(X^I)$
is simply connected.
\end{proof}

\vskip .1in

We apply the preceding corollary to the special case in which $|X| \simeq S^m$.  We are
implicitly using the fact that level-wise $\bZ/\ell$-completion of the $\cF$-space
$\ul\cB G^o(S^m)$ determines the Bousfield $\bZ/\ell$-completion of the associated
connective $\Omega$-spectrum (see \cite{Bar-Bo}).

\begin{prop}
\label{prop:S2-spec}
The homotopy type of the map $||\ul\cB G^o(S^m)_\ell||^+ \  \to \ ||\ul\cB G_\ell(S^m)||^+$
of $\Omega$-spectra is independent of $m>0$, and we denote this by 
\begin{equation}
\label{eqn:GS}
(\bf{BG^o(S)})_\ell \quad \to \quad   \bf{BG_\ell(S)}.
 \end{equation}
 The  sequence of $0$-connected $\Omega$-spectra
 \begin{equation}
\label{eqn:fib-seq}
(\bf{BG^o(S)}_{[0]})_\ell \quad \to \quad   \bf{BG_\ell(S)}_{[0]} \quad \to \quad {\bf K}(\bZ_\ell^*,1)
 \end{equation}
 is a fiber sequence of $\Omega$-spectra.
 \end{prop}

\begin{proof}
As shown in \cite[\S 4]{Segal}, the natural maps
$$\ul\cB G^o_\ell(S^m)({\bf 1}) \ = \ (\bigsqcup_{n\geq 0} B G^o_\ell(S^{mn}) \ \to \ (\bigsqcup_{n\geq 0} B G^o_\ell(S^{mn}))^+,$$
$$\ul\cB G_\ell(S^m)({\bf 1}) \ = \ (\bigsqcup_{n\geq 0} B G_\ell(S^{mn}) \ \to \ (\bigsqcup_{n\geq 0} B G_\ell(S^{mn}))^+$$
induce the maps $(-)\otimes_{\bZ/\ell[\bN]} \bZ/\ell[\bZ]$ in mod-$\ell$ homology and thus induce 
homotopy equivalences of $\ell$-complete spaces
\begin{equation}
\label{eqn:ell-compl}
\varinjlim_n B G^o_\ell(S^{n\cdot m}) \ \stackrel{\approx}{\to} \ (\bigsqcup_{n\geq 0} BG^o_\ell(S^{mn}))^+_{[0]}, \quad
\varinjlim_n B G_\ell(S^{n\cdot m}) \ \stackrel{\approx}{\to}  \ (\bigsqcup_{n\geq 0} B G_\ell(S^{mn}))^+_{[0]}.
\end{equation}
Observe that the natural maps (given by suspensions) 
$\varinjlim_n B G^o_\ell(S^n) \to \varinjlim_n B G^o_\ell(S^{n\cdot m})$ and 
$\varinjlim_n B G_\ell(S^n) \to \varinjlim_n B G_\ell(S^{n\cdot m})$ are isomorphisms, so that  
the homotopy type of the map of 0-connected spectra
$||\ul\cB G^o(S^m)_\ell||^+_{[0]} \  \to \ ||\ul\cB G_\ell(S^m)||_{[0]}^+$ is independent of $m$ for $m > 0$.

Since the map $G_\ell(S^n) \to G_\ell(S^{n+1})$ induced by suspension induces the
identity map $\bZ_\ell^* = \pi_0(G_\ell(S^n)) \ \to \  \pi_0(G_\ell(S^{n+1})) = \bZ_\ell^*$ for $n > 0$,
we conclude that 
$$ \varinjlim_n (B G^o(S^{n\cdot m}))_\ell \ \stackrel{\approx}{\to} \ 
\varinjlim_n B G^o_\ell(S^{n\cdot m}) \ \to \ \varinjlim_n B G_\ell(S^{n\cdot m})  \ \to \ K(\bZ_\ell^*,1)$$
is a fiber sequence.  This implies that 
$$||\ul\cB G^o(S^m)_\ell||^+_{[0]} \ \stackrel{\approx}{\to} \ 
||\ul\cB G_\ell^o(S^m_\ell)||^+_{[0]} \quad \to \quad ||\ul\cB G_\ell(S^m)||^+_{[0]} \quad \to  {\bf K}(\bZ_\ell^*,1)$$
is a fiber sequence of 0-connected $\Omega$-spectra as asserted.
\end{proof}

\vskip .1in

\begin{cor}
\label{cor:S-hom}
The homotopy groups \ $\pi_{i+1}((\bf{BG^o(S)})_\ell)  \ = \ \pi_{i+1}(\bf{BG_\ell(S)})$ \
can be naturally identified with the $\ell$-adic completion
of the $i$-th stable homotopy groups of the spheres for $i > 0$, whereas $\pi_1((\bf{BG^o(S)})_\ell)) = 0$
and $\pi_1(\bf{BG_\ell(S)}) = \bZ_\ell^*$.
\end{cor}

\vskip .2in


\section{Variants of the ${\bf J}$-map}
\label{sec:J-hom}

We shall find it convenient to refer to any $X$-fibration over $\ul \cB$
as an $S^2$-fibration if $|X|$ is homotopy equivalent to the 2-sphere.

\vskip .1in

The proof of the following lemma is essentially contained in its statement.

\begin{lemma}
\label{lem:cJ}
For each $n > 0$, the topological action of $GL_n(\bC)^{top}$ on $\bC^+$ determines an embedding
$GL_n(\bC)^{top} \hookrightarrow G^{top,o}(\bC^{n+})$, where $G^{top,o}(\bC^{n+})$
is the topological group of oriented self homotopy equivalences of $\bC^{n+}$.   For each $n > 0$ and 
each $I= (i_1,\ldots i_n) \in \ul\cN({\bf n})$, these maps induce maps of fibrations
\begin{equation}
\label{eqn:fiber-XI}
\xymatrix{
Sin({\ul \cB}(GL(\bC),\bC^+)^{top})_I \ar[r] \ar[d]_{\tau_{\bC^+}} 
\ar[r] \ar@/^2pc/[rr]
& \ul Sin(\cB (G^{top,o}(\bC^+),\bC^+))_I \ar[d]^{\pi_{\bC^+}^o} \ar[r] &
 \ul Sin(\cB (G^{top}(\bC^+),\bC^+))_I \ar[d]^{\pi_{\bC^+}}\\
Sin({\ul \cB}GL(\bC)^{top})_I \ar[r]^{\cJ^{o}} \ar@/_2pc/[rr]_{\cJ} & Sin(\ul\cB G^{top,o}(\bC^+))_I \ar[r]^{\iota} 
& Sin(\ul\cB G^{top}(\bC^+))_I.
}
\end{equation}
whose induced maps on fibers are homotopic to the identity map of $|X^I |$.

\end{lemma}

\vskip .1in

The following proposition follows immediately from Lemma \ref{lem:cJ}.

\vskip .1in

\begin{prop}
Let $X$ be a finite, pointed simplicial set with $|X| \simeq S^2$.  Then $\tau_{\bC^+}$ of (\ref{eqn:tau-top-spec})
fits in a commutative diagram of $\cF$-spaces
\begin{equation}
\label{eqn:cJ}
\xymatrix{
Sin({\ul \cB}(GL(\bC),\bC^+)^{top}) \ar[r] \ar[d]_{\tau_{\bC^+}} 
\ar[r] \ar@/^2pc/[rr]
& \ul Sin(\cB (G^{top,o}(\bC^+),\bC^+)) \ar[d]^{\pi_{\bC^+}^o} \ar[r] &
 \ul Sin(\cB (G^{top}(\bC^+),\bC^+)) \ar[d]^{\pi_{\bC^+}}\\
Sin({\ul \cB}GL(\bC)^{top}) \ar[r]^{\cJ^{o}} \ar@/_2pc/[rr]_{\cJ} & Sin(\ul\cB G^{top,o}(\bC^+)) \ar[r]^{\iota} & Sin(\ul\cB G^{top}(\bC^+))
}
\end{equation}
whose vertical maps are $X$-fibrations and whose squares are Cartesian.

Consequently, $\tau_{\bC^+}$ is an oriented $X$ fibration oriented by
$$adj^{-1}\circ \cJ^o: Sin({\ul \cB}(GL(\bC))^{top}) \ \dashrightarrow \  \ul\cB G^o(X),$$
where $adj: \ul\cB G(X) \to Sin(\ul\cB G(|X|)^{top})$ is the weak homotopy equivalence appearing in (\ref{eqn:compat-I}).
\end{prop}

\vskip .1in 

\begin{notation}
\label{note:J}
Let $X$ be a finite, pointed simplicial set with $|X| \simeq S^2$.  We denote by \quad \quad 
$ {\bf J^o}: {\bf B GL(C)} \quad \to \quad {\bf BG^o(S)}$ \quad \quad
the homotopy class of maps of connective $\Omega$-spectra 
determined by the map 
$$(adj^o)^{-1}\circ \cJ^o: Sin({\ul \cB}GL(\bC)^{top})  \dashrightarrow \ul\cB G^o(X) \ {\text in} \ Ho((\cF{\text -}spec/\ul\cN)_*)$$
in (\ref{eqn:cJ}).

We set \quad \quad 
${\bf J} \ \equiv \ \iota \circ  {\bf J^o}: \ {\bf B GL(C)} \quad \to \quad {\bf BG(S)}.$
\end{notation}

\vskip .1in

The following construction can be justified by using permutative categories whose objects and 
morphisms are schemes over $R$ for $R$ either $\bC$ or $W(k)$ or $k$, analogous to the justification
of Example \ref{ex:BGL}.  (See \cite{F80}.)  Alternatively, one can use the fact 
that $\tau^{top}_{|S^{2,alg}_\bC|}$ has an underlying algebraic structure.

\begin{construct}
\label{construct:base-change}
Let $R$ be either $\bC$ or $W(k)$ or $k$.
There is a natural  map of $\cF$-objects of pointed simplicial $R$-schemes
\begin{equation}
\label{eqn:tau-alg}
\tau_R: {\ul \cB}(GL_R,S_R^{2,alg}) \quad \to \quad {\ul \cB}GL_R \ .
\end{equation}
The evaluation at ${\bf 1} \in \cF$ of $\tau_R$ determines the natural projection 
$$ \coprod_{n\geq 0} B(GL_{R,n},S_R^{2n,alg}) \quad \to \quad \coprod_{n\geq 0} BGL_{R,n} \ .$$

Base changes along $\Spec \bC \to \Spec W(k)$ and along $\Spec k \to \Spec W$ determine a commutative
diagram of simplicial schemes
\begin{equation}
\label{eqn:base-change-s}
\xymatrix{
{\ul\cB}(GL_\bC,S^{2,alg}_\bC) \ar[d]_{\tau_{\bC}} \ar[r] & {\ul\cB}(GL_{W(k)},S^{2,alg}_{W(k)})  
\ar[d]^{\tau_{W(k)}}  & \ar[l] {\ul\cB}(GL_k,S^{2,alg}_k) \ar[d]^{\tau_k} \\
\ul\cB GL_\bC \ar[r] & \ul\cB GL_{W(k)} &  \ar[l] \ul\cB GL_k \ .
}
\end{equation}
\end{construct}

\vskip .1in

We complement Lemma \ref{lem:cJ} with the following fiber homotopy equivalences
of oriented $\bZ/\ell$-complete $S^2$-fibrations.

\begin{prop}
\label{prop:class-compare}
The following commutative diagram of $\cF$-spaces 
\begin{equation}
\label{eqn:C-to-Cwedge}
\xymatrix{
Sin( {\ul \cB} (GL(\bC),\bC^+)^{top})_\ell \ar[d]_{(\tau_{\bC^+})_\ell}  \ar[r]^{\tilde \gamma} &
Sin({\ul \cB} (GL(\bC),|S^{2,alg}_\bC|)^{top})_\ell \ar[r] \ar[d]^{(\tau_{|S^{2,alg}_\bC|})_\ell} &
({\ul\cB} (GL_{\bC},S_\bC^{2,alg}))^\wedge \ar[d]_{(\tau_\bC)^\wedge}    
 \\
 Sin({\ul \cB} GL(\bC)^{top})_\ell  \ar[r]^= & Sin(\ul\cB GL(\bC)^{top})_\ell \ar[r]_{f_{top\mapsto \bC}} &
({\ul\cB} GL_{\bC})^\wedge.
}
\end{equation}
consists of squares which are maps of oriented $\bZ/\ell$-complete $S^2$-fibrations with horizontal maps
which are weak equivalences.
Here, the lower horizontal arrow  $f_{top \mapsto \bC}$ is the homotopy class of $f_{top \mapsto \bC}:  
Sin({\ul \cB} GL(\bC))^{top})_\ell \ \dashrightarrow \ ({\ul\cB} GL_{\bC})^\wedge$ in $Ho(\cF[s.sets_*])$
specified by the functorial
``short chain"  of homotopy equivalences $\rho_{U\mapsto X}, \ \rho_{top \mapsto \pi_0}$ of Theorem \ref{thm:holim}.
\end{prop}

\begin{proof}
The compatibility of $\gamma_n\times \gamma_n$
with smash products (see (\ref{eqn:equiv})) together with the explicit descriptions of the $\cF$-spaces
(see Example \ref{ex:BGL}) provides the necessary functoriality with respect to 
maps in $\cF$  to yield the commutativity of the left square of (\ref{eqn:C-to-Cwedge}).
The compatibility established in Proposition \ref{prop:map-cone} of the actions of $GL_n(C)^{top}$ 
on $|S^{2n,alg}|$ and on $(\bC^n)^+$ utilizing the maps 
$ \gamma_n: |S^{2n,alg}_\bC| \to (\bC^n)^+$ of (\ref{eqn:tau-n}) imply that the left square of (\ref{eqn:C-to-Cwedge})
is a fiber homotopy equivalences.

The commutativity of the right square of (\ref{eqn:C-to-Cwedge}) follows from the naturality 
of $\rho_{U\mapsto X}, \ \rho_{top \mapsto \pi_0}$.
 This square is a homotopy equivalence of $S^2$-spaces by  
 the ``classical comparison theorem" in \'etale homotopy as in (\ref{eqn:Phi_X}). 
 
 By Corollary \ref{cor:relate-o}, $(\tau_{\bC^+})_\ell$ is oriented by $\lambda \circ ((adj^o)^{-1}\circ \cJ^o)_\ell$.
 We define the orientations of $(\tau_{|S^{2,alg}_\bC|})_\ell$ and $(\tau_\bC)^\wedge$ as those determined
 by $(\tau_{\bC^+})_\ell$: namely, \\
 $\lambda \circ ((adj^o)^{-1}\circ \cJ^o)_\ell$ and 
$\lambda \circ ((adj^o)^{-1}\circ \cJ^o)_\ell \circ (f_{top\mapsto \bC}^{-1}$).
\end{proof}

\vskip .1in

Construction \ref{construct:base-change} leads to the following maps of oriented $\bZ/\ell$-complete $X$-fibrations
where $|X| = S^2$.
Observe that $Sin(\ul\cB GL(\bC)^{top})$ is a pointed-connected special $\cF$-space over 
$\ul\cN$, so that $(\ul\cB GL_R)^\wedge$
is also a pointed-connected special $\cF$-space over $\ul\cN$ for $R$ equal to to $\bC$ or $k$ or $W(k)$.

\begin{prop}
\label{prop:tau-hat}
By applying $(-)^\wedge$ to (\ref{eqn:base-change-s}) in Construction \ref{construct:base-change} we 
obtain the following commutative diagram of $\cF$-spaces
\begin{equation}
\label{eqn:base-change-ss}
\xymatrix{
(\ul\cB (GL_\bC,S^{2,alg}_\bC))^\wedge \ar[d]_{(\tau_{\bC})^\wedge} \ar[r] & (\ul\cB (GL_{W(k)},S^{2,alg}_{W(k)}))^\wedge  
\ar[d]^{(\tau_{W(k)})^\wedge}  & \ar[l] (\ul\cB (GL_k,S^{2,alg}_k))^\wedge \ar[d]^{(\tau_k)^\wedge} \\
(\ul\cB GL_\bC)^\wedge \ar[r] & (\ul\cB GL_{W(k)})^\wedge &  \ar[l] (\ul\cB GL_k)^\wedge \ .
}
\end{equation}
whose squares constitute maps of oriented $\bZ/\ell$-complete $X$-fibrations with $|X| = S^2$ 
and whose horizontal maps are weak equivalences.

We denote by $f_{\bC\mapsto k}: (\ul\cB GL_\bC)^\wedge  \ \dashrightarrow \ (\ul\cB GL_k)^\wedge$  the homotopy class
of maps of $\cF$-spaces defined by the bottom row of (\ref{eqn:base-change-ss}).
\end{prop}

\begin{proof}
The fact that the $\cF$-spaces $({\ul \cB}GL_R)^\wedge$ and $({\ul \cB}(GL_R,S_R^{alg}))^\wedge$ 
are special $\cF$-spaces follows from Corollary \ref{cor:product}.  The fact that the lower horizontal maps of
(\ref{eqn:base-change-ss}) are homotopy equivalences follows from Proposition \ref{prop:comparison}.

A functorial isomorphism is given in \cite[Thm 10.7]{F82} relating the mod-$\ell$ \'etale cohomology 
of the geometric fiber of  a map $X \to Y$ of simplicial $R$-schemes such as 
those occurring in Corollary \ref{cor:product} and the mod-$\ell$ cohomology 
of the homotopy-theoretic fiber of $X^\wedge \to Y^\wedge$.
This, together with the fact that each $(BGL_{n,R})^\wedge$ is simply connected, provides the 
functorial homotopy comparisons of the fibers of $(\ul\cB(GL_R,S_R^{2,alg}))^\wedge_I \to 
(\ul\cB GL_R)^\wedge_I$ with $(\prod_{j= 1}^n S^{2i_j})_\ell$ for any $I = (i_1,\ldots,i_n) \in  \ul\cN({\bf n})$
required for $(\tau_R)^\wedge$ to be a $\bZ/\ell$-completed $S^2$-fibration.  
Consequently, the squares of (\ref{eqn:base-change-ss}) determine maps of $\bZ/\ell$-complete $X$-fibrations.

By Proposition \ref{prop:comparison}, the horizontal maps of (\ref{eqn:base-change-ss}) are weak equivalences.

We give $(\tau_{W(k)})^\wedge$ and $(\tau_k)^\wedge$ the orientations determined by the orientation
of $(\tau_{\bC})^\wedge$.
\end{proof}

\vskip .1in

\begin{defn}
\label{defn:cJ-k}
Let $X$ be a finite, pointed simplicial set with $|X| \simeq S^2$.
 We define 
 $$(\cJ^o_k)^\wedge: (\ul\cB GL_k)^\wedge \quad \dashrightarrow \quad \ul\cB G^o_\ell(X)$$
 to be the 
homotopy class of the classifying map for the oriented $Z/\ell$-completed $S^2$-fibration $(\tau_k)^\wedge$
occurring in (\ref{eqn:base-change-ss}).   
\end{defn}

\vskip .1in

 The following proposition foreshadows Theorem \ref{thm:stable-adams}.

\begin{prop}
\label{prop:relate-Js}
Set \ $f_{top \mapsto k} \ \equiv \ f_{top \mapsto \bC} \circ f_{\bC \mapsto k}$ in $Ho(\cF[s.sets_*]/\ul\cN)$,
and set 
$$\cJ^o_\ell \ \equiv \ (-)_\ell \circ (adj)^{-1}\circ \cJ: Sin({\ul \cB}GL(\bC)^{top})  \dashrightarrow \ul\cB G^o(X)  \ \to \ \ul\cB G^o_\ell(X).$$
Then the maps  
\begin{equation}
\label{eqn:equal-class}
\cJ_\ell^o, \ \  (\cJ_k)^\wedge \circ f_{top \mapsto k}:  Sin({\ul \cB} GL(\bC)^{top})_\ell \ \dashrightarrow \  \cB G_\ell(X)
\end{equation}
are equal in $Ho(\cF[s.sets_*]/\ul\cN)$.
\end{prop}

\begin{proof}
The asserted equality follows immediately from the observation that 
$(\cJ_k)^\wedge$ can be represented using Propositions \ref{prop:class-compare} and \ref{prop:tau-hat}
as the composition of maps in $Ho((\cF{\text -}spec/\ul\cN)_*)$
\begin{equation}
\label{eqn:string}
(\ul\cB GL_k)^\wedge \ \stackrel{f_{\bC\mapsto k}}{\dashleftarrow} \   (\ul\cB GL_\bC)^\wedge  \
 \stackrel{f_{top\mapsto \bC}}{\dashleftarrow} Sin({\ul \cB} GL(\bC)^{top})_\ell \
\stackrel{ \cJ^o_\ell}{\dashrightarrow} \  \ul\cB G_\ell(X).
\end{equation}
\end{proof}

\vskip .2in


\section{Adams operations}
\label{sec:Adams}

We refer to \cite[\S 3.2]{Atiyah} for a general discussion of the Adams operations and to \cite[11.13]{JFA}
for their stable version.  
We view the Adams operation $\psi^p$ as a map ${\bf kU } \to {\bf kU}$ of connective $\Omega$-spectra.
In Theorem \ref{thm:Sul}, we give  our interpretation of the models for $(\psi^p)_\ell$ utilized by D. Sullivan 
in \cite[\S 4]{Sul} and envisioned by D. Quillen in \cite{Quillen68}.

\begin{thm} 
\label{thm:Sul}
Let $\sigma_{\bC} \in Gal(\bC/\bQ)$ be an extension to a discontinuous automorphism of $\bC$ of the Teichm\"uller lifting
$\sigma_{W(k)}: W(k) \to W(k)$ of the arithmetic Frobenius $\sigma \in Gal(k/\bF_p)$.  
 The following commutative diagram in $\cF[s.sets_*]$ (with horizontal maps which are
homotopy equivalences)
\begin{equation}
\label{eqn:sigma} 
\xymatrix{
Sin(\ul \cB GL(\bC)^{top})_\ell  \ar[rr]^-{f_{top\mapsto \bC}}  & & (\ul\cB GL_\bC)^\wedge \ar[d]^{(\sigma_\bC^{-1})^\wedge} 
\ar[rr]^-{f_{\bC \mapsto k}} & & ({\ul\cB}GL_k)^\wedge \ar[d]^{F^\wedge} \\
Sin(\ul \cB GL(\bC)^{top})_\ell  \ar[rr]_-{f_{top\mapsto \bC}} & & (\ul\cB GL_{\bC})^\wedge  \ar[rr]^-{f_{\bC \mapsto k}} & & ({\ul\cB}GL_k)^\wedge 
}
\end{equation}
determines a homotopy commutative diagram of associated connective spectra upon applying the functor $(||-||)$
\begin{equation}
\label{eqn:sigma-||} 
\xymatrix{
||Sin(\ul \cB GL(\bC)^{top})_\ell || \ar[d]_{(\psi^p)_\ell} \ar[rr]^-{f_{top\mapsto \bC}}  & & ||(\ul\cB GL_\bC)^\wedge|| 
\ar[d]^{(\sigma_\bC^{-1})^\wedge} 
\ar[rr]^-{f_{\bC \mapsto k}} & & ||({\ul\cB}GL_k)^\wedge||^{+} \ar[d]^{F^\wedge} \\
||(Sin(\ul \cB GL(\bC)^{top})_\ell ||  \ar[rr]_-{f_{top\mapsto \bC}} & & ||(\ul\cB GL_{\bC})^\wedge||^{+}  \ar[rr]^-{f_{\bC \mapsto k}} 
& & ||({\ul\cB}GL_k)^\wedge||
}
\end{equation}
\end{thm}

\begin{proof}
We consider the commutative diagram of $\cF$-spaces over $\ul \cN$:
\begin{equation}
\label{eqn:T-GL}
\xymatrix@C6pc{
Sin(\ul\cB T(\bC)^{top})_\ell   \ar@/_4pc/[ddd]_{(-)^p} \ar[d] \ar[r]^{f_{top\mapsto \bC}} & (\ul\cB T_{\bC})^\wedge \ar[d] \ar@/^4pc/[ddd]^{(-)^p} \\
Sin(\ul \cB GL(\bC)^{top})_\ell  \ar[d]^{(\psi^p)_\ell} \ar[r]^{f_{top\mapsto \bC}} & (\ul\cB GL_{\bC}))^\wedge \ar[d]_{(\sigma_\bC^{-1})^\wedge}\\
Sin(\ul \cB GL(\bC)^{top})_\ell   \ar[r]^{f_{top\mapsto \bC}} & (\ul\cB GL_{\bC}))^\wedge \\
Sin(\ul\cB T(\bC)^{top})_\ell \ar[r]^{f_{top\mapsto \bC}}  \ar[u] & (\ul\cB T_{\bC})^\wedge \ar[u]
}
\end{equation}
(where $T_n(\bC) \hookrightarrow GL_n(\bC)$ is the subgroup of diagonal matrices).
The splitting theorem for (complex, topological) $K$-theory  implies that to  prove the homotopy commutativity
of the left square of (\ref{eqn:sigma-||}) (which is the same as the middle square of (\ref{eqn:T-GL})),
 it suffices to show that the larger
square determined by the upper and middle squares of (\ref{eqn:T-GL}) is commutative.  This is 
implied by the commutativity of outer square of (\ref{eqn:T-GL}).

The commutativity of the right squares of (\ref{eqn:sigma}) and (\ref{eqn:sigma-||}) is implied
by Proposition \ref{prop:Frob}.
\end{proof}

The following proposition presents a first version of the stable Adams conjecture.

\vskip .1in
 
 \begin{prop}
 \label{prop:F-wedge}
 As in Proposition \ref{prop:hom-fiber}, the (geometric) Frobenius map determines the commutative square of $\cF$-spaces
  \begin{equation}
\label{eqn:Fwedge}
\xymatrix{
(\ul\cB (GL_k,S_k^{2,alg}))^\wedge \ar[r]^{F^\wedge} \ar[d]_{(\tau_k)^\wedge} &  
(\ul\cB (GL_k,S_k^{2,alg}))^\wedge \ar[d]^{(\tau_k)^\wedge} \\
(\ul\cB GL_k)^\wedge \ar[r]^{F^\wedge}  & (\ul\cB GL_k)^\wedge,
}
\end{equation}
which constitutes a map of  $\bZ/\ell$-complete $S^2$-fibrations
over the weak equivalence $F^\wedge: (\ul\cB GL_k)^\wedge \to (\ul\cB GL_k)^\wedge$.

Consequently, the maps
$$ \cJ_k^\wedge, \ \circ \cJ_k^\wedge \circ F^\wedge: \ (\ul\cB GL_k)^\wedge \quad \to \quad \ul\cB G_\ell^o(S^2)
\quad \to \ul\cB G_\ell(S^2)$$
are homotopic.
\end{prop}

 \begin{proof}
 Proposition \ref{prop:hom-fiber}  implies for each $I = (i_1,\ldots,i_n) \in \ul\cN({\bf n})$ that
 \begin{equation}
 \label{eqn:Fwedge-I}
 \xymatrix{
((\ul\cB (GL_k,S_k^{2,alg}))^\wedge)_I \ar[r]^{F^\wedge} \ar[d]_{(\tau_k)^\wedge} &  
((\ul\cB (GL_k,S_k^{2,alg}))^\wedge)_I \ar[d]^{(\tau_k)^\wedge} \\
((\ul\cB GL_k)^\wedge)_I \ar[r]^{F^\wedge}  & ((\ul\cB GL_k)^\wedge)_I ,
}
\end{equation}
induces a map between homotopy fibers of the vertical maps of (\ref{eqn:Fwedge-I})
which is equivalent to $\times_{i=1}^n p^{n_i}: \prod S^{2n_i}_\ell \to\prod S^{2n_i}_\ell$.  
Since the image of $p$  is a unit in $\bZ/\ell$, each of these squares is a 
fiber homotopy equivalence over the homotopy equivalence
$((\ul\cB GL_k)^\wedge)_I \to ((\ul\cB GL_k)^\wedge)_I $.  This implies the first assertion.

Thus, (\ref{eqn:Fwedge}) implies that the $\bZ/\ell$-complete $S^2$ fibrations
$(\tau_k)^\wedge , \ (F^\wedge)^*( (\tau_k)^\wedge)$ are equal in $X_\ell{\text -}fib^o((\ul\cB GL_k)$. 
Hence, the classifying maps
for $(\tau_k)^\wedge$ and $(F^\wedge)^*( (\tau_k)^\wedge)$ are equal in 
$Hom_{Ho((\cF{\text -}spec/\ul\cN)_*)}((\ul\cB GL_k)^\wedge, \ul\cB G_\ell(X))$.   
This conclusion is equivalent to the second assertion granted that $ \cJ_k^\wedge$
is the classifying map for $(\tau_k)^\wedge$.
 \end{proof}

\vskip .2in


\section{The Stable Adams Conjecture}
\label{sec:stable}

\vskip .1in

In our proof of Theorem \ref{thm:stable-adams}, we  use  the following proposition ``commuting"
the functors $(\bZ_\ell)_\infty(-)$ with $(||-||)^+$.  We implicitly use properties of $\cF$-spaces
summarized in Theorem \ref{thm:summary}.

\begin{prop}
\label{prop:ell-spectra}
Let $\ul \cB$ be a pointed-connected special $\cF$-space 
over $\ul\cN$ with the property that $\ul\cB_I$ is $\ell$-good for
every $I \in \ul\cN(\bf n), \ n > 0$.   Then the natural map of $\cF$-spaces
\ $\ul\cB  \  \to \ \ul \cB_\ell$ \ induces a homotopy equivalence of connective $\Omega$-spectra
\begin{equation}
\label{eqn:specequiv}
\bZ/\ell_\infty(|| \ul\cB ||^+) \times_{{\bf K(\bZ_\ell,0)}} {\bf K(\bZ,0)} \quad \to \quad || \ul \cB_\ell ||^+,
\end{equation}
where $\bZ/\ell_\infty(|| \ul\cB ||^+)$ is the Bousfield $\ell$-completion of the connective $\Omega$-spectrum $|| \ul\cB ||^+$.
\end{prop}

\begin{proof}

Because  $\ul\cB_I$ is $\ell$-good for every $I \in \ul\cN(\bf n), \ n > 0$, $\ul\cB (T_n)$ is $\ell$-complete where
$T_n$ is a finite simplicial set with $|T_n| \simeq S^n$ (where the functor $\ul\cB$ sending pointed finite sets
to pointed simplicial sets is extended to finite simplicial sets by taking diagonals).  Consequently, $||\ul\cB_\ell ||$
is a level-wise $\ell$-completion of $|| \ul\cB ||$.

    As shown in \cite{Bar-Bo}, we may replace the Bousfield localization $\bZ/\ell_\infty((|| \ul\cB ||^+)_{[0]}) $
 of the 0-connected spectrum $(|| \ul\cB ||^+)_{[0]}$
by the level-wise $\bZ/\ell$-completion $(|| \ul\cB_\ell ||^+)_{[0]}$, where $(-)_{[0]}$ sends a connective $\Omega$-spectrum
to its 0-connected cover. Consider the diagram of connective $\Omega$-spectra
\begin{equation}
 \label{eqn:Omega-diag}
 \xymatrix{
( \bZ/\ell_\infty(|| \ul\cB ||^+)_{[0]} \ar[r]^=  \ar[d] & (\bZ/\ell_\infty(|| \ul\cB ||^+)_{[0]} \ar[d] &  \ar[l]_-{\approx} \ar[d] (|| \ul \cB_\ell ||^+)_{[0]} \\
 \bZ/\ell_\infty(|| \ul\cB ||^+) \times_{{\bf K(\bZ_\ell,0)}} {\bf K(\bZ,0)} \ar[d]  \ar[r] & \bZ/\ell_\infty(|| \ul\cB ||^+) \ar[d] & ||\ul \cB_\ell ||^+   \ar[l] \ar[d] \\
 {\bf K(\bZ,0)} \ar[r] & {\bf K(\bZ_\ell,0)} & \ar[l] {\bf K(\bZ,0)}
 }
\end{equation}
whose columns are fiber sequences.  We conclude the homotopy equivalence of (\ref{eqn:specequiv}) by
observing that the map $||\ul \cB_\ell ||^+ \to \bZ/\ell_\infty(|| \ul\cB ||^+)$ factors through \\
$\bZ/\ell_\infty(|| \ul\cB ||^+) \times_{{\bf K(\bZ_\ell,0)}}  {\bf K(\bZ,0)}$.
\end{proof}

\vskip .1in

We now verify that Theorem \ref{thm:X-ell-universal} together and
the translation of Proposition \ref{prop:F-wedge} into the context of connective $\Omega$-spectra
imply the following theorem.  We utilize ${\bf BG_\ell(S)}$ to denote the connective $\Omega$-spectrum  $||\ul\cB G_\ell(S^2)||^+$.

\vskip .1in

\begin{thm}
\label{thm:stable-adams}
Let ${\bf kU}$ denote the connective $\Omega$-spectrum of (topological) complex K-theory.
If $p$ and  $\ell$ are distinct primes, then the maps 
$${\bf J}_\ell, \ {\bf J}_\ell \circ (\psi^p)_\ell:  ({\bf kU})_\ell \times_{{\bf K(\bZ_\ell,0)}} {\bf K(\bZ,0)} 
\quad \to \quad {\bf BG_\ell(S)}$$
are equal in the homotopy category of  connective $\Omega$-spectra.
\end{thm} 

\begin{proof}
By Proposition \ref{prop:ell-spectra} applied to $\ul\cB$ equal to $(\ul\cB GL(\bC))^{top}_\ell$, to prove the asserted homotopy 
equivalence of spectra it suffices to prove that the compositions
\begin{equation}
\label{eqn:step1}
 (\cJ)_\ell  , \  (\cJ)_\ell \circ (\psi^p)_\ell: (\ul\cB GL(\bC))^{top}_\ell  \
  \to \ \ul\cB G_\ell(S^2)
\end{equation}
are homotopic.  

We consider the following diagram of pointed--connected special $\cF$-spaces
 \begin{equation}
 \label{eqn:st-ad}
 \xymatrix{
Sin(\ul \cB GL(\bC)^{top})_\ell \ar@/_4pc/[dd]_{\cJ_\ell} \ar[d]_{(\psi^p)_\ell} \ar[rr]^{f_{top\mapsto k}} & 
& (\ul\cB GL_k)^\wedge \ar@/^4pc/[dd]^{\cJ_k} \ar[d]^{F^\wedge} \\
Sin(\ul \cB GL(\bC)^{top})_\ell \ar[d]_{\cJ_\ell}  \ar[rr]^{f_{top\mapsto k}} & & (\ul\cB GL_k)^\wedge \ar[d]^{(\cJ_k)^\wedge} \\
Sin(\ul \cB G^{top}(\bC^+))_\ell  \ar[rr]_{\approx} & &  \ul\cB G_\ell(S^2).
}
\end{equation}
The top square of (\ref{eqn:st-ad}) homotopy
commutes by Theorem \ref{thm:Sul}; the bottom square homotopy commutes by Proposition \ref{prop:relate-Js};
the far right triangle homotopy commutes by Proposition \ref{prop:F-wedge} and Theorem \ref{thm:X-ell-universal}.
Thus, the homotopy commutativity of the far left triangle is implied by these homotopy commutative squares
and  the fact that the horizontal
maps of (\ref{eqn:st-ad}) are homotopy equivalences.
\end{proof}

\vskip .1in

We recall from  Proposition \ref{prop:S2-spec} the notation $(\bf{BG^o(S)})_\ell \ \to \   \bf{BG_\ell(S)}$
for the map of connective $\Omega$-spectra $||(\ul\cB G^o(S^m))_\ell||^+ \  \to \ ||\ul\cB G_\ell(S^m)||^+$
(which is independent of $m$).  If $\bf{B}$ is an $\Omega$-spectrum, we denote by $\bf{B}_{[0]}$ the
0-connected cover of $\bf{B}$.

\begin{thm}
\label{thm:oriented-stable-adams}
Let ${\bf J}_\ell^o:  ({\bf BU})_\ell \ \to \ (\bf{BG^o(S)}_{[0]})_\ell$ denote the map 
of 1-connected $\Omega$-spectra induced by ${\bf J}_\ell$.
If $p$ and $\ell$ are distinct primes, then the following maps 
  $${\bf J}_\ell^o, \ {\bf J}_\ell^o \circ (\psi^p)_\ell:  ({\bf BU})_\ell 
\quad \to \quad (\bf{BG^o(S)}_{[0]})_\ell $$
are equal in the homotopy category of connective $\Omega$-spectra.
\end{thm}

\begin{proof}
The fiber sequence $ (\bf{BG^o(S)}_{[0]})_\ell \ \to  \ \bf{BG_\ell (S)}_{[0]} \ \to  \ {\bf K(\bZ_\ell^*,1)}$
of (\ref{eqn:GS}) determines the exact sequence
$$Hom_{Ho(c.\Omega-spectra)}(({\bf BU})_\ell,{\bf K(\bZ_\ell^*,0)})  \to Hom_{Ho(c.\Omega-spectra)}(({\bf BU})_\ell,(\bf{BG^o(S)}_{[0]})_\ell)$$
$$ \quad \to \quad Hom_{Ho(c.\Omega-spectra)}(({\bf BU})_\ell,\bf{BG_\ell(S)}_{[0]}).$$
Consider the difference  
$${\bf J}_\ell^o - {\bf J}_\ell^o \circ (\psi^p)_\ell \in Hom_{Ho(c.\Omega-spectra)}(({\bf BU})_\ell,(\bf{BG^o(S}_{[0]})_\ell).$$
By Theorem \ref{thm:stable-adams}, this difference maps to 0 in 
$Hom_{Ho(c.\Omega-spectra)}(({\bf BU})_\ell,\bf{BG_\ell(S)}_{[0]})$ and thus is an element of 
$Hom_{Ho(c.\Omega-spectra)}(({\bf BU})_\ell,{\bf K(\bZ_\ell^*,0)}) $.  Because 
 $({\bf BU})_\ell$ is 0-connected, $Hom_{Ho(c.\Omega-spectra)}(({\bf BU})_\ell,{\bf K(\bZ_\ell^*,0)}) = 0$ so that
$${\bf J}_\ell^o \ =  \ {\bf J}_\ell^o \circ (\psi^p)_\ell \ \in \ Hom_{Ho(c.\Omega-spectra)}(({\bf BU})_\ell,(\bf{BG^o(S)}_{[0]})_\ell)$$ 
as asserted.
\end{proof}

\vskip .1in

\begin{remark}
\label{rem:false}
The maps  ${\bf J}_\ell, \ {\bf J}_\ell \circ (\psi^p)_\ell$ of Theorem \ref{thm:oriented-stable-adams}  
factor through $\bf{BG_\ell^o(S)} \to \bf{BG_\ell(S)}$.  However, as explained in \cite{B-K}, these 
factorizations are not equal in\\
 $Hom_{Ho(c.\Omega-spectra)}(({\bf kU})_\ell \times_{\bf {K(\bZ_\ell,0)}} {\bf K(\bZ,0)},(\bf{BG_\ell^o(S)})$.
\end{remark}

Applying Theorem \ref{thm:oriented-stable-adams}, we easily obtain the following $\ell$-local version of the Stable Adams Conjecture .

\vskip .1in

\begin{cor}
\label{cor:local}
If $p$ and $\ell$ are distinct primes, the following maps of 1-connected spectra
$${\bf J}_{(\ell)}^o, \ {\bf J}_{(\ell)}^o \circ (\psi^p)_{(\ell)}:  ({\bf BU})_{(\ell)} 
\quad \to \quad (\bf{BG^o(S)}_{[0]})_{(\ell)} $$
are equal in the homotopy category of  connective $\Omega$-spectra.
\end{cor}

\begin{proof}
We refer to \cite[\S II.13] {JFA} for the $\ell$-localized stable Adams operation, 
$(\psi^p)_{(\ell)}: ({\bf BU})_{(\ell)} \to ({\bf BU})_{(\ell)}$ and observe that the 
$\bZ/\ell$-completions of the maps  ${\bf J}_{(\ell)}^o, \ {\bf J}_{(\ell)}^o \circ (\psi^p)_{(\ell)}$
of $\ell$-local $\Omega$-spectra
are homotopic to the maps of $\bZ/\ell$-complete $\Omega$-spectra ${\bf J}_\ell^o, \ {\bf J}_\ell^o \circ (\psi^p)_\ell.$  
The computations of homotopy groups recalled in Proposition \ref{prop:fingen}
imply that 
\ $(\bf{BG^o(S)}_{[0]})_{(\ell)} \ \to \ (\bf{BG^o(S)}_{[0]})_\ell $
is a weak equivalence (see \cite[VI.68]{Bo-Kan})), so that the corollary follows from Theorem \ref{thm:oriented-stable-adams}
and the commutativity of the following square of $\Omega$-spectra
\begin{equation}
\label{eqn:com-spectra}
\xymatrix@C6pc{
({\bf BU})_{(\ell)}  \ar[d] \ar[r]^-{{\bf J}_{(\ell)}^o - {\bf J}_{(\ell)}^o \circ (\psi^p)_{(\ell)}} & (\bf{BG^o(S)}_{[0]})_{(\ell)} \ar[d] \\
({\bf BU})_\ell \ar[r]_-{{\bf J}_{\ell}^o - \ {\bf J}_{\ell}^o \circ (\psi^p)_{\ell}} & (\bf{BG^o(S)}_{[0]})_\ell .
}
\end{equation}
\end{proof}

\vskip .1in

\begin{remark}
\label{rem:quat}
As in  \cite[Thm 10.5]{F82}, the techniques of this paper apply to prove quaternionic analogs
of the preceding results.  One should replace $GL_n(\bC)$ acting on $\bC^{n+}$ by $Sp_{2n}(\bC)$
acting on $Q^{n+} \simeq S^{4n}$, where $Q^n$ is an $n$-dimensional vector space over the 
division algebra of quaternions (which has dimension 4 over $\bR$).  We leave the details to the reader.
\end{remark}


\vskip .2in


\begin{thebibliography}{20} 

\bibitem{Adams} J.F. Adams, {\em On the groups $J(X)$. I.}, Topology {\bf 2} (1963), 181 - 195.

\bibitem{JFA} J.F. Adams, {\it Stable homotopy and generalized homology.}  Chicago Lecture
Series in Mathematics, University of Chicago Press, 1974.

%

\bibitem{A-M} M. Artin, B. Mazur, {\it \'Etale Homotopy}. Lecture Notes in Mathematics {\bf 100},
Springer-Verlag, 1969.

\bibitem{SGA4} M. Artin, A. Grothendieck, and J.-L Verdier, {\it Theorie des Topos et 
Cohomologies \'Etale des schemas}, Lecture Notes in Mathematics {\bf 269} Springer-Verlag, 1972.

\bibitem{Atiyah}  M. Atiyah, {\it K-theory}, W.A. Benjaman, New York - Amsterdam, 1967

\bibitem{Bar-Bo} T. Barthel and A.K. Bousfield, {\em On the comparison of stable and unstable $p$-completion},
Proc. Amer. Amath. Soc. {bf 147} (2019), 897 - 908.


\bibitem{B-K} P. Bhattacharya, N. Kitchloo, {\em The $p$-local Stable Adams Conjecture: Erratum},
 arXiv:1803.11014.
 

\bibitem{Bo} A.K. Bousfield, {\em The localization of spectra with respect to homology}, Topology {\bf 18} (1979), 257-281.

\bibitem{Bo-F} A.K. Bousfield, E. Friedlander, {\em Homotopy theory of $\Gamma$-spaces, spectra, and 
bisimplicial sets}.  Lecture Notes in Mathematics {\bf 658}, Springer-Verlag, 1978.

\bibitem{Bo-Kan} A.K. Bousfield, D. Kan, {\it Homotopy limits, completions, and localizations}.  Lecture
Notes in Mathematics {\bf 569} Springer-Verlag, 1977.

\bibitem{Deligne} P. Deligne, {\it Cohomologie \'Etale (SGA $4\frac{1}{2}$)}, Lecture
Notes in Mathematics {\bf 304} Springer-Verlag, 1972.

\bibitem{F73} E. Friedlander, {\em Fibrations in \'etale homotopy theory}, Publications Math IHES {\bf 42} (1973), 
5 - 46.



\bibitem{F80}  E. Friedlander, {\em The infinite loop Adams conjecture via classification theorems for $\cF$-spaces}, 
Math. Proc. Camb. Phil. Soc.{\bf 87} (1980), 109 - 150.
%

\bibitem{F82} E. Friedlander, {\it \'Etale Homotopy of Simplicial Schemes}, Annals of Math. Stud 
{\bf 104}, Princeton Univ. Press, 1982.

\bibitem{F-Par} E. Friedlander and B. Parshall, {\em \'Etale cohomology of reductive groups}, 
{\it Algebraic $K$-Theory, Evanston,1980}, Lecture Notes in Math. {\bf 854} (1981), Springer-Verlag, 127 - 140.


%

\bibitem{MacL} S. Mac Lane, S. {\it Categories for the working Mathematician}, Grad. Texts in Math {\bf 5},
Springer-Verlag, 1974.

\bibitem{May67} J.P. May, {\it Simplicial objects in algebraic topology}, Van Nostrand Mathematical 
Studies, {\bf 11}, 1967.

\bibitem{May75} J.P. May, {\it Classifying spaces and fibrations},  Mem. Amer. Amth. Soc {\bf 155} (1975).


\bibitem{May} J.P. May, {\em The spectra associated to permutative categories}, Topology {\bf 17} (1978), 225 - 228.


\bibitem{Mc-Seg} D. McDuff, G. Segal, {\em Homology fibrations and the ``group completion" theorem},
Invent Math {\bf 31} (1975), 274 - 284.

\bibitem{Milne} J. Milne, {\it \'Etale Cohomology}, Princeton Mathematical Series {\bf 33}, Princeton 
Univ. Press, 1980.

\bibitem{nLab}  https://ncatlab.org/nlab/show/proper+model+category.


\bibitem{Quick} G. Quick, {\em Profinite homotopy theory}, Doc Math {\bf 13} (2008), 585 - 612.

\bibitem{Quick2} G. Quick, {Some remarks on profinite completion of spaces}, Adv. Stud. Pure Math {\bf 63} (2012, 413 - 448.

\bibitem{Quillen68} D. Quillen, {\em Some remarks on \'etale homotopy theory and a conjecture of Adams},
Topology {\bf 7} (1968), 111 - 116.
%
\bibitem{Quillen} D. Quillen, {\em The Adams Conjecture}, Topology {\bf 10} (1971), 67 - 80.

\bibitem{Segal} G. Segal, {\em Categories and cohomology theories}, Topology {\bf 13} (1974), 293 - 312.

\bibitem{Serre} J.-P. Serre, {\em Groupes d'homotopie et classes de groupes abelian}, Ann. of Math {\bf 58} (1953), 
258 - 294.

\bibitem{Stash} J.D. Stasheff, {A classification theorem for fiber spaces}, Topology {\bf 2}(1963), 239 - 246.

\bibitem{Sul} D. Sullivan, {\em Genetics of homotopy theory and the Adams Conjecture}, Ann. of Math
{\bf 100} (1974), 1 - 79.

\end{thebibliography}
\end{document}